\documentclass[11pt,letterpaper]{article}

\usepackage{amsmath,amssymb,amsthm}
\usepackage{MnSymbol} 
\usepackage{graphicx,color}
\usepackage[hyphens]{url}
\usepackage{dsfont}
\usepackage{booktabs}
\usepackage[normalem]{ulem}
\usepackage{mathtools}
\usepackage[hidelinks,colorlinks=true,linkcolor=blue,citecolor=blue]{hyperref}
\usepackage[nameinlink,capitalize]{cleveref}
\usepackage{multirow}
\usepackage{algorithm}
\usepackage[noend]{algpseudocode}
\usepackage{caption}
\usepackage{subcaption}
\usepackage{enumitem}
\usepackage{soul} 
\usepackage[square,sort,numbers]{natbib}
\usepackage[sort,nocompress,noadjust]{cite}

\usepackage{tabularx}
\newcolumntype{L}{>{\raggedright\arraybackslash}X}

\numberwithin{equation}{section}
\newtheorem{theorem}{Theorem}[section]
\newtheorem{cor}[theorem]{Corollary}
\newtheorem{lemma}[theorem]{Lemma}
\newtheorem{remark}[theorem]{Remark}
\newtheorem{fact}[theorem]{Fact}

\newtheorem{obs}[theorem]{Observation}

\newtheorem{example}[theorem]{Example}

\newtheorem{defin}[theorem]{Definition}

\newcommand{\cF}{\mathcal{F}}

\newcommand{\cM}{\mathcal{M}}

\newcommand{\cQ}{\mathcal{Q}}
\newcommand{\cR}{\mathcal{R}}

\newcommand{\R}{\mathbb{R}}
\newcommand{\C}{\mathbb{C}}

\newcommand{\Prob}{\mathbb{P}}
\newcommand*{\E}{\mathbb{E}}
\newcommand{\supp}{\operatorname{supp}}

\newcommand*{\eps}{\varepsilon}

\newcommand*{\kap}{\kappa}

\renewcommand{\leq}{\leqslant}
\renewcommand{\geq}{\geqslant}

\providecommand{\abs}[1]{\left\lvert#1\right\rvert}

\DeclareMathOperator{\Racc}{\mathrm{R_{acc}}}  
\DeclareMathOperator{\Cov}{\mathrm{Cov}}
\DeclareMathOperator{\Var}{\mathrm{Var}}
\DeclareMathOperator*{\asc}{\overset{\text{a.s.}}{\longrightarrow}}
\DeclareMathOperator*{\ase}{\overset{\text{a.s.}}{=}}

\newcommand{\Arc}[2]{\mathrm{Arcsine}(#1,#2)}
\DeclareMathOperator{\ArcmM}{{\mathrm{Arcsine}}(m,M)}
\DeclareMathOperator{\ArcmMt}{{\mathrm{Arcsine}}(\tilde{m},\tilde{M})}

\DeclareMathOperator{\FlippedArcmM}{\mathrm{FlippedArcsine}(m,M)}
\DeclareMathOperator{\FlippedArcmMt}{\mathrm{FlippedArcsine}(\tilde{m},\tilde{M})}

\newcommand{\limn}{\lim_{n \to \infty}}
\DeclareMathOperator{\calM}{\cM}

\usepackage[margin=1in]{geometry}

\makeatletter
\def\blfootnote{\gdef\@thefnmark{}\@footnotetext}
\makeatother

\begin{document}

	\title{Acceleration by Random Stepsizes:  \\ Hedging, Equalization, and the Arcsine Stepsize Schedule} 
	
	\author{
		Jason M. Altschuler
		\\	UPenn \\	\texttt{alts@upenn.edu}
		\and
		Pablo A. Parrilo \\
		LIDS - MIT \\	\texttt{parrilo@mit.edu}
	}
	\maketitle

    \blfootnote{This result first appeared in Chapter 6 of the 2018 MS Thesis~\citep{altschuler2018greed} of the first author (advised by the second author). This was the first convergence rate for any stepsize schedule (deterministic or randomized) that improved asymptotically over the standard unaccelerated rate for constant-stepsize GD, in any setting beyond quadratics.\label{fn:history}}

\begin{abstract}
    We show that for separable convex optimization, random stepsizes fully accelerate Gradient Descent. Specifically, using inverse stepsizes i.i.d.\ from the Arcsine distribution improves the convergence rate from $O(\kappa)$ to $O(\kappa^{1/2})$, where $\kappa$ is the condition number. No momentum or other algorithmic modifications are required. Our starting point is a remarkable ``equalization property'' of the Arcsine distribution: it yields an identical convergence rate for all quadratic functions. A key technical insight is that martingale arguments extend this phenomenon to all separable convex functions. We interpret this equalization as an extreme form of hedging: by using this random distribution over stepsizes, Gradient Descent converges at exactly the same rate for all functions in the function class.
\end{abstract}


\section{Introduction}\label{sec:intro} 

It is classically known that by using certain deterministic stepsize schedules $\{\alpha_t\}$, the Gradient Descent algorithm (GD)
\begin{align}
	x_{t+1} = x_t - \alpha_t \nabla f(x_t), \qquad t=0,1,2,\dots
	\label{eq-intro:gd}
\end{align}
solves convex optimization problems $\min_{x \in \R^d} f(x)$ to arbitrary precision from an arbitrary initialization~$x_0$. This result can be found in any textbook on convex optimization, e.g.,~\citep{boyd,bubeck-book,nesterov-survey,polyakbook,hazan2016introduction,luenberger1984linear,BertsekasNonlinear,lan2020first}. The present paper investigates the following question:
\begin{align}
	\text{Is it possible to accelerate the convergence of GD by using \emph{random} stepsizes?}
	\label{eq-intro:motivation}
\end{align}
This is different from the traditional approach for  obtaining accelerated convergence, which modifies the GD algorithm beyond just stepsizes (e.g., using momentum~\citep{nesterov-agd}); see the survey~\citep{d2021acceleration}. 

Our main result is that for separable convex optimization, the answer is yes. Specifically, we provide an i.i.d.\ stepsize schedule that makes GD converge at the optimal accelerated rate. See \S\ref{ssec:intro:cont} for a precise statement; here we first contextualize with the relevant literature on random stepsizes.

\paragraph*{The special case of quadratics.} This question is understood (only) for optimizing convex \emph{quadratic} functions. In particular, the line of work~\citep{pronzato11,pronzato13,kalousek} shows that for quadratic functions~$f$ that are $m$-strongly convex and $M$-smooth (i.e., $m I \preceq \nabla^2 f \preceq MI$), there is a unique optimal distribution for i.i.d.\ stepsizes that makes GD converge at a rate of
\begin{align}
	\left( \frac{\|x_n - x^*\|}{\|x_0 - x^*\|} \right)^{1/n}  \longrightarrow\; \frac{\sqrt{\kappa} - 1}{\sqrt{\kappa} + 1}\,,
	\label{eq-intro:rate}
\end{align}
where $x^*$ denotes the minimizer of $f$, $\kappa := \tfrac{M}{m}$ denotes the condition number, and this convergence is in an appropriate probabilistic sense (e.g., almost-sure or $L^1$). 
This stepsize distribution is most naturally stated in terms of the inverse stepsizes
$\beta_t := 1/\alpha_t$ having the Arcsine distribution on the interval $(m,M)$, i.e., 
density
\begin{align}
	\frac{d \mu}{d\beta} = \frac{\mathds{1}_{\beta \in (m,M)}}{\pi \sqrt{(M-\beta)(\beta-m)}}  \,.
	\label{eq-intro:arcsine}
\end{align}
Briefly, this Arcsine distribution arises naturally for the special case of quadratics because: (1) the convergence rate is invariant w.r.t.\ the stepsize order and thus depends only on the empirical measure of the stepsizes; and (2) the Arcsine distribution is the limit as $n \to \infty$ of the empirical measure of the optimal $n$ (deterministic) stepsizes for GD~\citep{young53}. See \S\ref{ssec:hi:quad} for details.

The resulting convergence rate~\eqref{eq-intro:rate} requires roughly $\Theta(\sqrt{\kappa} \log \tfrac{1}{\eps})$ iterations to produce an $\eps$-accurate solution. Such a rate is called ``accelerated'' because its dependence on $\kappa$ improves over the standard $\Theta(\kappa \log \tfrac{1}{\eps})$ convergence rate required by GD with constant stepsizes. Moreover, this convergence rate $\Theta(\sqrt{\kappa} \log \tfrac{1}{\eps})$ is information-theoretically optimal for any first-order algorithm~\citep{nesterov-survey,nem-yudin}, and matches the worst-case convergence rate of accelerated algorithms that change GD beyond just its stepsizes, e.g., the Conjugate Gradient method~\citep{hestenes1952methods}, Polyak's Heavy Ball method~\citep{polyak1964some}, Nesterov-style momentum methods~\citep{nesterov-agd}, and more~\citep{d2021acceleration}. 

\begin{figure}
	\centering
	\includegraphics[width=0.6\linewidth]{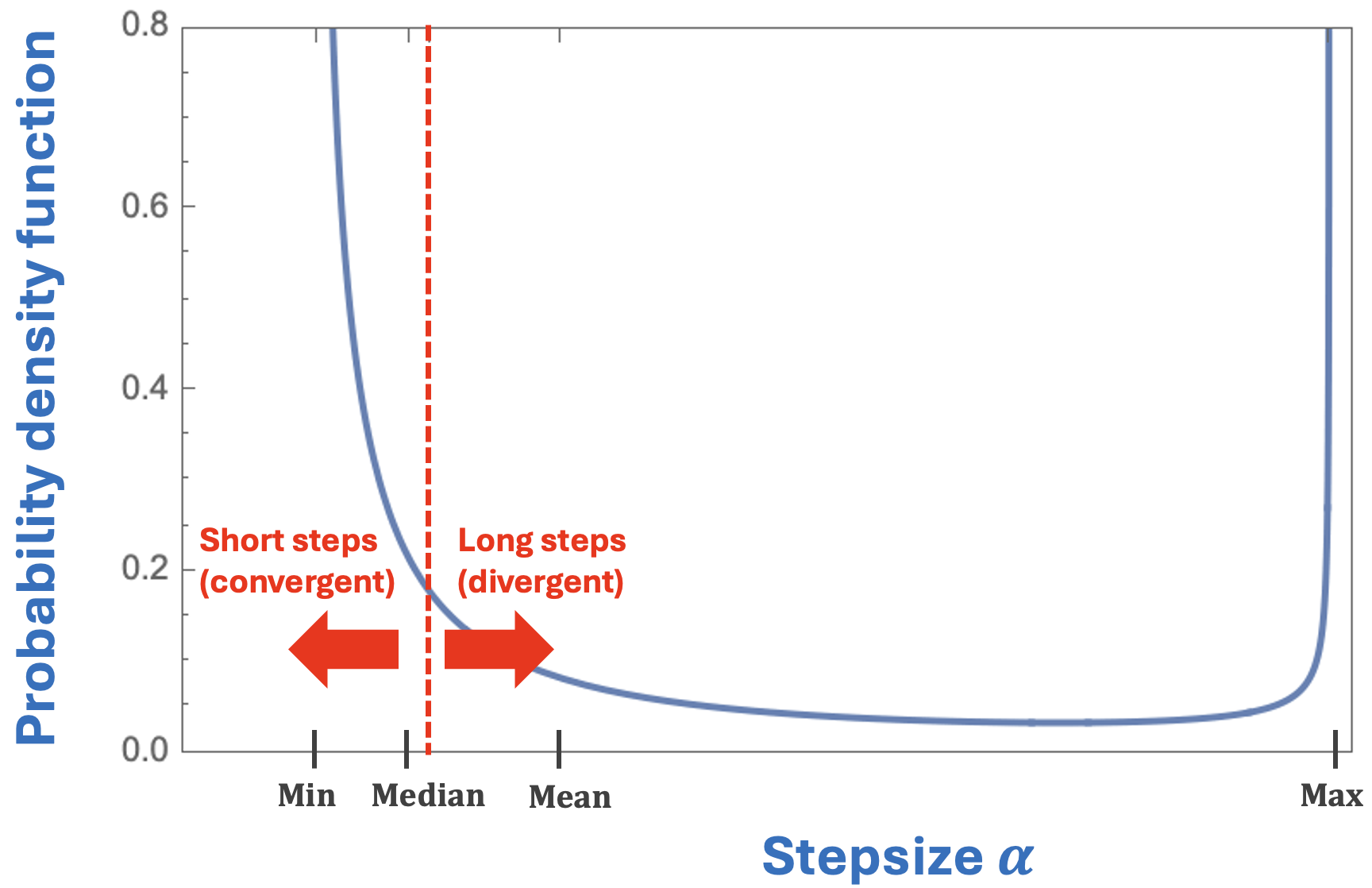}
	\caption{\footnotesize The induced distribution of stepsizes $\alpha$, for inverse stepsizes $\alpha^{-1}$ taken from the Arcsine distribution on $(m,M)$. 
			The minimum stepsize is $1/M$, the maximum is $1/m$, and the median is $2/(M+m)$, which is the optimal value for constant stepsize schedules. For constant stepsize schedules, the dashed red line $\bar{\alpha} = 2/M$ is the threshold for convergence; larger stepsizes $\alpha$ lead to divergence, and shorter stepsizes $\alpha$ lead to slow convergence. 
			This distribution optimally \emph{hedges} between short and long steps. Note that the mean stepsize is $1/\sqrt{Mm}$ which is larger than the divergence threshold for constant stepsize schedules, by a factor of $\Theta(\sqrt{\kappa})$. This plot sets $\kappa = 10$, $m=1/\kappa$, $M=1$; the discrepancy between this distribution and standard constant stepsizes is even more dramatic for larger $\kappa$.
		 }
	\label{fig:distribution}
\end{figure}

\paragraph*{Beyond quadratics?} For general convex optimization, momentum-based algorithms still achieve these optimally accelerated rates $\Theta(\sqrt{\kappa} \log \tfrac{1}{\eps})$~\citep{nesterov-agd}. However, existing acceleration results for random stepsizes do not extend beyond the special case of quadratic functions $f$---otherwise $\nabla f$ is non-linear, making GD a non-linear map, which completely breaks all previous analyses. This failure of the analyses is due to fundamental differences between quadratic optimization and convex optimization:
\begin{itemize}
	\item The optimal deterministic stepsizes for quadratic optimization~\citep{young53} (which the random stepsizes are built upon) can lead to divergent behavior for convex optimization~\citep{altschuler2018greed}.
	\item The convergence rate is invariant w.r.t.\ the stepsize order for quadratic optimization (which motivates the opportunity for random stepsizes), but this is false for convex optimization~\citep{altschuler2018greed}.
	\item Using time-varying stepsizes is equivalent to using momentum for quadratic optimization (which enables GD to achieve accelerated rates), but this is false for convex optimization~\citep{d2021acceleration}.
\end{itemize}
Due to these challenges, it remained unclear whether random stepsizes could lead to the accelerated rate $\Theta(\sqrt{\kappa} \log \tfrac{1}{\eps})$ for convex optimization. In fact, it was unknown even if random stepsizes could provide an arbitrary small improvement over the $\Theta(\kappa \log \tfrac{1}{\eps})$ rate of constant-stepsize GD in any setting beyond quadratic optimization.

\subsection{Contribution}\label{ssec:intro:cont}

This paper shows that random stepsizes provably accelerate GD beyond the setting of quadratics. Specifically, our main result (Theorem~\ref{thm:sep:main}) shows that GD converges at the optimally accelerated rate for separable convex functions. This separability property is defined below and, for $C^2$ functions~$f$, is equivalent to commutativity of the Hessians $\{\nabla^2 f(x)\}_{x \in \R^d}$, see Appendix~\ref{app:sep-commute}.

\begin{defin}[Separable functions]\label{def:sep}
	A function $f : \R^d \to \R$ is \emph{separable} if it admits a decomposition of the form
	\begin{align}
		f(x) = \sum_{i=1}^d f_i([Ux]_i)
		\label{eq:def-separability}
	\end{align} 
	for some orthogonal matrix $U \in \R^{d \times d}$ and some functions $f_1, \dots, f_d : \R \to \R$.
\end{defin}

\begin{theorem}[Random stepsizes accelerate GD for separable convex optimization]\label{thm:sep:main}
	Consider any dimension $d$, any separable function $f : \R^d \to \R$ that is $m$-strongly convex and $M$-smooth, and any initialization point $x_0$ that is not equal to the minimizer $x^*$ of $f$. By using i.i.d.\ inverse stepsizes $\alpha_t^{-1}$ from the Arcsine distribution~\eqref{eq-intro:arcsine}, GD converges at the rate
	\begin{align}
		\left( \frac{\|x_n - x^*\|}{\|x_0 - x^*\|} \right)^{1/n} \longrightarrow \;\; \frac{\sqrt{\kappa} - 1}{\sqrt{\kappa} + 1}\,,
	\end{align}
	where this convergence is in the almost sure and $L^1$ sense. Moreover, this is the unique distribution for which GD achieves this optimal rate. 
\end{theorem}

\begin{table}[t]
	\footnotesize
	\centering
	\begin{tabular}{|c|c|c|c|}
		\hline
		\textbf{Stepsizes$\;\;$\textbackslash$\;\;$Function class} & \textbf{Quadratic }        & \textbf{Separable}                  & \textbf{Convex}                              \\ \hline
		\textbf{Constant}        & $\Theta(\kappa)$ folklore   & $\Theta(\kappa)$ folklore  & $\Theta(\kappa)$ folklore    \\
		\textbf{Deterministic (Chebyshev/Silver)}   & $\Theta(\kappa^{1/2})$~\citep{young53} & $O(\kappa^{\log_{1+\sqrt{2}} 2})$~\citep{alt23hedging1}  & $O(\kappa^{\log_{1+\sqrt{2}} 2})$~\citep{alt23hedging1}   \\
		\textbf{Random (Arcsine)}    & $\Theta(\kappa^{1/2})$~\citep{kalousek,pronzato11,pronzato13} & \color{blue}{ $\Theta(
			\kappa^{1/2})$ Theorem~\ref{thm:sep:main}} & unknown if randomness helps
			\\ \hline
	\end{tabular}
    
	\caption{\footnotesize{Convergence rates for minimizing $\kappa$-conditioned functions using GD~\eqref{eq-intro:gd} with different stepsize choices. For brevity, we omit the dependence on the accuracy as it is always $\log 1/\eps$.
    The main result is in blue: for separable convex functions, random stepsizes enable accelerated convergence at the rate $\kappa^{1/2}$, which is known to be information-theoretically optimal~\citep{nem-yudin}. 
    The asymptotic notation $O$ refers to an upper bound; $\Theta$ refers to matching upper and lower bounds. See the future work in \S\ref{sec:future} for conjectures motivated by our result, about potential gaps between deterministic and randomized stepsizes for GD, and gaps between separable and general convex optimization for GD.
	}}
	\label{tab:comparison}
\end{table}

\subsection{Discussion}

\paragraph*{Optimality.} The rate in Theorem~\ref{thm:sep:main} is optimal since there is an exactly matching lower bound: there is an explicit $\kappa$-conditioned quadratic function for which any (Krylov) first-order algorithm cannot contract to the optimum faster than this rate $\tfrac{\sqrt{\kappa}-1}{\sqrt{\kappa}+1}$~\citep[Theorem 2.1.12]{nesterov-survey}.

\paragraph*{Relation with other stepsize schedules.} 
Here we provide a brief historical perspective to contextualize the random stepsizes studied in this paper.
As mentioned earlier, the classic result of~\citep{young53} showed that GD can be accelerated for the special case of quadratics by using deterministic stepsizes related to Chebyshev polynomials. Since then, many alternative stepsize schedules have been investigated, some adaptive, e.g., exact line search~\citep{polyakbook}, Armijo-Goldstein schedules~\citep{nesterov-survey}, Polyak-type schedules~\citep{polyakbook}, and Barzilai-Borwein-type schedules~\citep{barzilai1988two}. While some of these schedules can be quite effective in practice, none has led to an analysis for any setting beyond quadratics that outperforms the ``unaccelerated'' contraction factor 
$
(\tfrac{\|x_n - x^*\|}{\|x_0 - x^*\|})^{1/n} \leq \tfrac{\kappa - 1}{\kappa + 1}$
of constant stepsize GD, which corresponds to a convergence rate of $\Theta(\kappa \log \tfrac{1}{\eps})$. Starting with~\citep{altschuler2018greed}, a recent line of work has shown partially accelerated convergence rates for general convex optimization using deterministic stepsize schedules suggested by semidefinite-programming-based estimates of convergence rates~\citep{altschuler2018greed,daccache2019performance,eloi2022worst,gupta22,grimmer23,alt23hedging1,alt23hedging2,GrimmerShuWang,GSW24,wang2024relaxed,zhang2024accelerated,grimmer2024composing,bok24,bok2025optimized}, see the recent survey~\citep{altpar26expository}. In particular,~\citep{alt23hedging1} showed the silver convergence rate
$O(\kappa^{\log_{1 + \sqrt{2}} 2} \log \tfrac{1}{\eps}) \approx O(\kappa^{0.78} \log \tfrac{1}{\eps})$ by using a certain fractal-like deterministic stepsize schedule called the silver stepsizes. This rate is conjecturally optimal among all deterministic stepsize schedules~\citep{alt23hedging1}. 
Theorem~\ref{thm:sep:main} is incomparable: it shows a fully accelerated rate $O(\kappa^{1/2} \log \tfrac{1}{\eps})$, but under the additional assumption of separability. See Table~\ref{tab:comparison}. This incomparability raises interesting questions about the possible benefit of randomization for GD, see the discussion of future work in~\S\ref{sec:future}.

\paragraph*{Generality.} The phenomenon of acceleration via random stepsizes in Theorem~\ref{thm:sep:main} extends to other structural assumptions on $f$ and other notions of convergence (see \S\ref{sec:extensions}), and to settings where GD is performed inexactly (see \S\ref{ssec:discussion:inexact}). Note also that since GD does not depend on the choice of the basis $U$ in the separable decomposition~\eqref{eq:def-separability}, Theorem~\ref{thm:sep:main} only requires the existence of some such $U$ (which is the definition of separability). In particular, the algorithm need not know $U$.

\paragraph*{Variability.}
Randomized algorithms have random trajectories.
The best run can be significantly better than the typical run (whose rate converges almost surely to the accelerated rate, see \S\ref{ssec:discussion:hp}) and the worst run (whose rate can diverge but only with exponentially small probability, see \S\ref{ssec:discussion:notions}).

\paragraph*{Game-theoretic connections and lower bounds.} This paper considers random stepsizes in order to accelerate GD. This randomized perspective is also helpful for the dual problem of constructing lower bounds, i.e., exhibiting functions for which GD converges slowly. In particular, lifting to \emph{distributions} over functions leads to a game between an algorithm (who chooses a distribution over stepsizes) and an adversary (who chooses a distribution over functions) which is symmetric in that the Arcsine distribution is the optimal answer for both. This leads to new perspectives on classical lower bounds for GD; details in \S\ref{ssec:discussion:lb}.

\paragraph*{The equalization property and logarithmic potential theory.} A distinguishing feature of our analysis is a connection to potential theory: the variational characterization for the optimal stepsize distribution mirrors the variational characterization for a certain equilibrium distribution which configures a unit of charged particles so as to minimize the logarithmic potential energy. The optimal distribution satisfies a remarkable ``equalization property'' which in the physical context amounts to a constant potential over space, and in the optimization context amounts to an identical convergence rate over all quadratic functions. A key technical insight is that martingale arguments extend this phenomenon to all separable convex functions. Links between Krylov-subspace algorithms and potential theory are classical~\citep{6steps}; we point out that these connections extend beyond quadratic objectives and to random stepsizes. 
See \S\ref{sec:hi} for a conceptual overview of our analysis techniques and how we exploit this equalization property to prove Theorem~\ref{thm:sep:main}.

\subsection{Organization}

\S\ref{sec:hi} provides a conceptual overview of our analysis approach.~\S\ref{sec:separable} proves our main result, Theorem~\ref{thm:sep:main}.~\S\ref{sec:extensions} discusses the generality of our result.~\S\ref{sec:discussion} provides auxiliary results about notions of convergence, stability, and game-theoretic interpretations.~\S\ref{sec:future} describes several directions for future work that are motivated by our results. For brevity, several proofs are deferred to the appendix.

\section{Conceptual overview: acceleration by random stepsizes}\label{sec:hi}

Here we provide informal derivations of our main result in order to overview the key conceptual ideas. We begin by briefly recalling in \S\ref{ssec:hi:quad} how random stepsizes can accelerate GD for the special case of quadratic optimization. Then in \S\ref{ssec:hi:convex}, we describe the challenges for acceleration beyond quadratics, and our new techniques for circumventing these challenges.

\subsection{Quadratic optimization}\label{ssec:hi:quad}

Consider running GD on the class $\cQ_{m,M}$ of quadratic functions $f$ that are $m$-strongly convex and $M$-smooth. 
What is the optimal distribution from which to draw i.i.d.\ stepsizes? In order to understand this, we first recall the optimal deterministic stepsizes. 

\paragraph*{Optimal deterministic stepsizes.} Without loss of generality after translating the solution $x^* = 0$, we have
\begin{align}
	f(x) = \frac{1}{2} x^T H x\, \qquad \text{ where } \qquad mI \preceq H \preceq MI\,.
	\label{eq:hi-quad:obj}
\end{align}
Since the function is quadratic, its gradient is linear $\nabla f(x) = Hx$, hence GD is a linear map
\begin{align}
	x_{t+1} = x_t - \alpha_t \nabla f(x_t) = (I - \alpha_t H) x_t\,,
	\label{eq:hi-quad:alg}
\end{align}
and therefore the final iterate is
\begin{align}
	x_n = \prod_{t=0}^{n-1} (I - \alpha_t H) x_0\,.
	\label{eq:hi-quad:poly}
\end{align}
The convergence rate for a given stepsize schedule---in the worst case over objective functions in $\cQ_{m,M}$ and initializations $x_0$---is therefore
\begin{align}
	\sup_{f \in \cQ_{m,M}, \, x_0 \neq x^*}
	\left( \frac{\|x_n - x^*\|}{\|x_0 - x^*\|} \right)^{1/n}
	&=
	\sup_{mI \preceq H \preceq MI, \, x_0 \neq x^*}
	\left( \frac{\|\prod_{t=0}^{n-1} (I - \alpha_t H) x_0\|}{\|x_0\|}  \right)^{1/n} \nonumber \\
	&=
	\sup_{mI \preceq H \preceq MI}
	\| \prod_{t=0}^{n-1} (I - \alpha_t H) \|^{1/n}
	\label{eq:hi-quad:rate-mid}
	\\
	&=
	\sup_{\lambda \in [m,M]} \prod_{t=0}^{n-1} |1 - \alpha_t \lambda|^{1/n}\,.
	\label{eq:hi-quad:rate}
\end{align}
Above, the penultimate step is by definition of the spectral norm, and the final step is by diagonalizing.
It is a simple and celebrated fact, dating back to Young in 1953~\citep{young53}, that minimizing this convergence rate~\eqref{eq:hi-quad:rate} over stepsizes is equivalent to a certain extremal polynomial problem, and that as a consequence, the optimal stepsizes minimizing~\eqref{eq:hi-quad:rate} are the inverses of the roots of a (translated and scaled) Chebyshev polynomial, in any order.
Said more precisely: for any horizon length $n$, the optimal $n$ stepsizes are any permutation of 
\begin{align}
	\left\{ \left( \frac{M+m}{2} + \frac{M-m}{2} \cos \left( \frac{2t+1}{2n}\pi \right) \right)^{-1} \right\}_{t=0}^{n-1}
	\label{eq:quad-cheb}
\end{align}
and the corresponding $n$-step convergence rate in~\eqref{eq:hi-quad:rate} is
\begin{align}
	\sup_{f \in \cQ_{m,M}, \, x_0 \neq x^*}
	\left( \frac{\|x_n - x^*\|}{\|x_0 - x^*\|} \right)^{1/n}
	=
    \left[ 2\frac{(\sqrt{\kappa} + 1)^n(\sqrt{\kappa} - 1)^n}{(\sqrt{\kappa} + 1)^{2n} + (\sqrt{\kappa} - 1)^{2n}} \right]^{1/n}
\overset{n \to \infty}{\longrightarrow}
\frac{\sqrt{\kappa} - 1}{\sqrt{\kappa} + 1}\,.
	\label{eq:quad-rate}
\end{align}

\paragraph*{Optimal random stepsizes.} Much less well-known is the fact that this accelerated rate 
can also be asymptotically achieved by simply using i.i.d.\ stepsizes from an  appropriate distribution~\citep{kalousek,pronzato11,pronzato13}. 
The intuition is that the convergence rate~\eqref{eq:hi-quad:rate} is invariant in the order of the stepsizes, instead depending only on their empirical measure, and as $n \to \infty$ one can achieve the same empirical measure (and thus the same convergence rate) by replacing the carefully chosen deterministic Chebyshev stepsize schedule~\eqref{eq:quad-cheb} with i.i.d.\ stepsizes from their limiting empirical measure. See~\citep[Chapter 3]{altschuler2018greed}
for a discussion of this through the lens of exchangeability and De Finetti's Theorem.

\par Concretely, if the inverse stepsizes $\beta_t := 1/\alpha_t$ are i.i.d.\ from some distribution $\mu$, then (forgoing measure-theoretic technicalities) the Law of Large Numbers ensures that the $n$-step convergence rate in~\eqref{eq:hi-quad:rate} converges almost surely and in expectation to
\begin{align}
	\sup_{f \in \cQ_{m,M}, \, x_0 \neq x^*}
	\left( \frac{\|x_n - x^*\|}{\|x_0 - x^*\|} \right)^{1/n}
	&=
	\sup_{\lambda \in [m,M]} \prod_{t=0}^{n-1} |1 - \alpha_t \lambda|^{1/n} \nonumber
	\\ &=
	\exp\left( \sup_{\lambda \in [m,M]} \frac{1}{n} \sum_{t=0}^{n-1} \log |1 - \beta_t^{-1} \lambda| \right) \nonumber
	\\ &\overset{n \to \infty}{\longrightarrow}
	\exp\left( \sup_{\lambda \in [m,M]} \E_{\beta \sim \mu}\log |1 - \beta^{-1} \lambda| \right) 
	\,
	\label{eq:quad-rand-rate}
\end{align}
This informal derivation suggests that the distribution $\mu$ minimizing the asymptotic convergence rate~\eqref{eq:quad-rand-rate} should be the appropriate limit as $n \to \infty$ of the empirical measure for the inverse Chebyshev stepsize schedule (this limit is the $\ArcmM$ distribution, since $a + b\cos (\Theta)$ has the Arcsine distribution on $[a-b,a+b]$ if $\Theta$ is uniformly distributed on $[0,\pi],$ see Figure~\ref{fig:cheby}), and moreover the optimal value should be the accelerated rate $\frac{\sqrt{\kappa}-1}{\sqrt{\kappa}+1}$. This intuition is formalized by the following lemma.

\begin{figure}
	\centering
	\begin{subfigure}{.32\textwidth}
		\centering
		\includegraphics[width=0.9\linewidth]{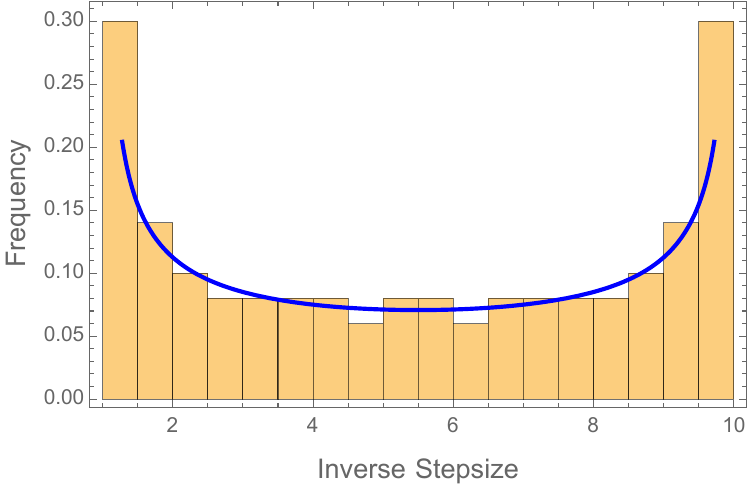}
		\caption{$n=10^2$}
		\label{fig:cheby1}
	\end{subfigure}
	\begin{subfigure}{.3\textwidth}
		\centering
		\includegraphics[width=0.9\linewidth]{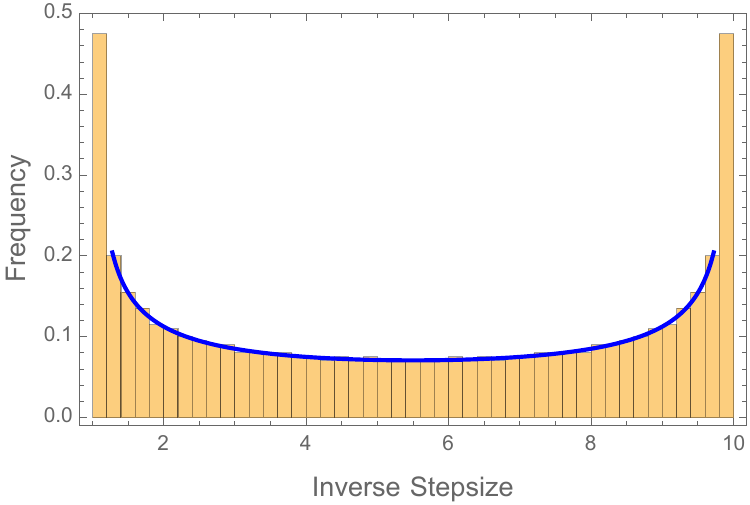}
		\caption{$n=10^3$}
		\label{fig:cheby2}
	\end{subfigure}
		\begin{subfigure}{.3\textwidth}
		\centering
		\includegraphics[width=0.9\linewidth]{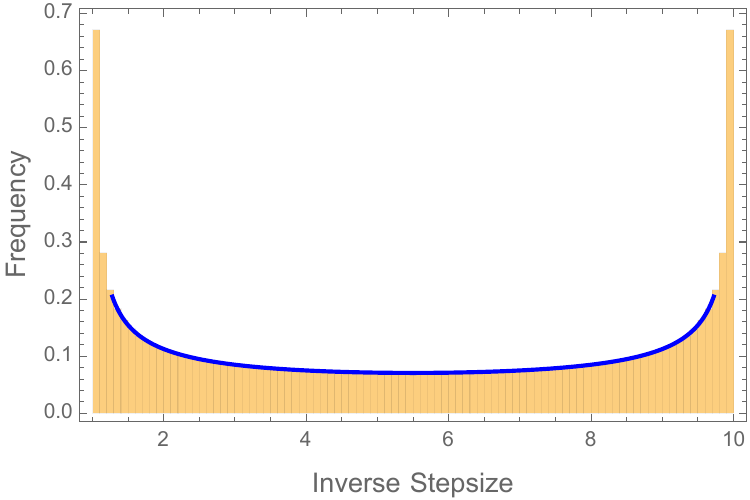}
		\caption{$n=10^4$}
		\label{fig:cheby3}
	\end{subfigure}
	\caption{\footnotesize As $n \to \infty$, the empirical distribution (yellow) of the $n$-step inverse Chebyshev stepsize schedule~\eqref{eq:quad-cheb} converges to the Arcsine distribution (blue). The fit is increasingly accurate as $n$ increases from $10^2$ (left) to $10^4$ (right). Demonstrated here for $m=1$ and $M=10$.}
	\label{fig:cheby}
\end{figure}

\begin{lemma}[Extremal property of the Arcsine distribution]\label{lem:arc-extremal}
	The extremal problem
	\begin{align}
		\inf_{\mu} \sup_{\lambda \in [m,M]} \E_{\beta \sim \mu} \log \abs{1 - \beta^{-1} \lambda }
		\label{eq:extremal}
	\end{align}
	has optimal value $\log \frac{\sqrt{\kappa} - 1}{\sqrt{\kappa} + 1}$, achieved by the unique optimal solution $\mu = \ArcmM$.
\end{lemma}

This lemma was proven in the optimization literature via a long, tour-de-force integral calculation~\citep{kalousek}. In Appendix~\ref{app:pot}, we present an alternative, short proof from Altschuler's thesis~\citep{altschuler2018greed} that argues by connecting this extremal problem for stepsize distributions to a classical extremal problem in potential theory about electrostatic capacitors. The key feature of this connection is the following \emph{equalization property} of the Arcsine distribution. 

\begin{lemma}[Equalization property of the Arcsine distribution]\label{lem:arc-equalization}
	For all $\lambda \in [m, M]$,
	\begin{align}
	\E_{\beta \sim \ArcmM} \log \abs{1 - \beta^{-1}\lambda} =
	\log \frac{\sqrt{\kappa} - 1}{\sqrt{\kappa} + 1}\,.
	\label{eq:equalization}
	\end{align}
In particular, this value does not depend on $\lambda$.
\end{lemma}

This equalization property uniquely defines the Arcsine distribution in that there is no other distribution satisfying~\eqref{eq:equalization}; see Appendix~\ref{app:pot}. In words, this property shows that the Arcsine distribution ``equalizes'' the convergence rate over all quadratic functions (parameterized by their curvatures $\lambda$). We emphasize that this is not just an algebraic identity: this equalization property also has a natural interpretation in terms of a game between an optimization algorithm (who chooses a stepsize to minimize the convergence rate) and an adversary (who chooses a quadratic function to maximize the convergence rate), see \S\ref{ssec:discussion:lb}. In Appendix~\ref{app:pot} we elaborate on the connections to potential theory, and leverage that connection to provide short proofs of these lemmas. For pedagogical reasons and to provide further intuition, in Appendix~\ref{app:equalizing-alternative} we also provide a short semi-rigorous proof using only elementary single-variable calculus.

\subsection{Beyond quadratic optimization}\label{ssec:hi:convex}

We now turn to accelerating GD beyond the special case of quadratic functions $f$. The central challenge is that the curvature, as measured by the Hessian $\nabla^2 f$, is no longer constant. This breaks the above quadratic analysis in several ways since it is no longer possible to summarize or re-parameterize the Hessians along the GD trajectory in terms of a single scalar $\lambda \in [m,M]$. To explain these issues, let us attempt to adapt this analysis to the general convex setting, assuming for simplicity here that $f$ is twice differentiable (our actual analysis in \S\ref{sec:separable} avoids any such assumption).

\paragraph*{Time-varying curvature.} Clearly, if $f$ is not quadratic, then the gradient $\nabla f$ is not linear, and as a consequence the GD map is not linear. Assuming as before that $x^* = 0$ without loss of generality, one can generalize the linear GD map~\eqref{eq:hi-quad:poly} to the non-quadratic setting as
\begin{align}
	x_n = \prod_{t=0}^{n-1} (I - \alpha_t H_t) x_0 \qquad \text{where} \qquad H_t := \int_0^1 \nabla^2 f( ux_t) du\,,
	\label{eq:hi-convex:poly}
\end{align}
since by the Fundamental Theorem of Calculus, $\nabla f(x_t) = H_tx_t$ and thus $x_{t+1} = x_t - \alpha_t \nabla f(x_t) = (I - \alpha_t H_t) x_t$. 
The interpretation of $H_t$ is that it is the average Hessian between $x^* = 0$ and $x_t$, i.e., it measures the relevant local curvature in the $t$-th iteration of GD. In the quadratic setting, $H_t = H$ for all $t$ since the Hessian $\nabla^2f \equiv H$ is constant, which recovers~\eqref{eq:hi-quad:poly} from~\eqref{eq:hi-convex:poly}. However, for the general non-quadratic setting, the curvature $H_t$ varies in $t$. 
This is the overarching challenge mentioned above. For the purpose of accelerating GD, this challenge manifests in two interrelated ways: the lack of \emph{commutativity} and \emph{predictability} of the curvature $H_t$ along the trajectory of GD. 
Below we describe both challenges and how we circumvent them.

\subsubsection{Issue 1: commutativity (deriving the extremal problem)} 
The first issue is that for convex functions, the Hessians at different points do not necessarily commute. Thus, in general, $\{H_t\}$ do not commute and therefore are not simultaneously diagonalizable. This precludes the key step~\eqref{eq:hi-quad:rate} in the quadratic analysis above, which decomposes the convergence rate across eigendirections, thereby reducing the (multivariate) problem of convergence analysis to an analytically tractable (univariate) extremal problem.
This is the only reason we assume separability: it is equivalent to commuting Hessians (Lemma~\ref{lem:sep-commute}).
Concretely, separability allows us to show an analog of~\eqref{eq:hi-quad:rate} in which the convergence rate of a stepsize schedule $\{\alpha_t\}$ is bounded by
\begin{align}
	\prod_{t=0}^{n-1} |1 - \alpha_t \lambda_t|^{1/n}\,,
	\label{eq:hi-convex:rate}
\end{align}
the difference to the quadratic case being that there all $\lambda_t \in [m,M]$ are the same, whereas now $\lambda_t \in [m,M]$ can be time-varying and can depend on all previous stepsizes $\{\alpha_s\}_{s < t}$. 
\par By taking an appropriate\footnote{A subtle but important technical issue is that one cannot simply use the Law of Large Numbers to derive the asymptotic convergence rate from the $n$-step rate (c.f.,~\eqref{eq:quad-rand-rate} for how that worked in the quadratic case). Indeed, in the convex case here, the relevant empirical sum $\frac{1}{n} \sum_{t=0}^{n-1} \log |1 - \beta_t^{-1} \lambda_t|$ has non-i.i.d.\ summands: they are neither identically distributed (since $\lambda_t$ is time-varying) nor independent (since $\lambda_t$ can depend on previous stepsizes). Briefly, we overcome the identically distributed issue by using the equalization property of the Arcsine distribution (Lemma~\ref{lem:arc-equalization}) to show that these summands still have the desired expectation, and we overcome the independence issue by using a martingale-based argument to show that the fresh randomness in $\alpha_t$ makes these summands ``independent enough'' to prove a specially tailored version of Kolmogorov's Strong Law of Large Numbers for this setting. Details in \S\ref{ssec:sep:uni}.} limit of the $n$-step convergence rate~\eqref{eq:hi-convex:rate}, it can be shown that the (logarithm of the) asymptotic rate for i.i.d.\ inverse stepsizes $\beta_t := \alpha_t^{-1}$ from a distribution $\mu$ is governed by the extremal problem:
\begin{align}
	\inf_\mu \sup_{\{\mathcal{A}_t\}-\text{adapted process }\{\lambda_t \in [m,M] \}} \limsup_{n \to \infty} \frac{1}{n} \sum_{t=0}^{n-1} \E_{\beta_t \sim \mu} \log \abs{1 - \beta_t^{-1} \lambda_t}
	\label{eq:extremal-convex}
\end{align}
where $\{\mathcal{A}_t\}$ denotes the natural filtration corresponding to the inverse stepsizes $\{\beta_t\}$.  Observe that~\eqref{eq:extremal-convex} recovers the extremal problem~\eqref{eq:quad-rand-rate} for the quadratic case if $\lambda_t$ are constant, but in general is different because $\lambda_t$ can be time-varying.
This brings us to the second issue.

\subsubsection{Issue 2: unpredictability (solving the extremal problem)}

The second issue is that the curvature is time-varying (as measured by $\lambda_t$), thus less predictable, which makes it difficult to choose stepsizes that lead to fast convergence (as measured by the extremal problem~\eqref{eq:extremal-convex}). 
The key upshot of our random stepsizes is that the randomness of $\beta_t$ prevents an adversarial choice of $\lambda_t \in [m,M]$ since $\lambda_t$ must be ``chosen'' before $\beta_t$ is drawn. More precisely, we exploit the equalization property of the Arcsine distribution (Lemma~\ref{lem:arc-equalization}) to argue that for \emph{any choice of $\lambda_t$}, the distribution of $\beta_t$ ensures that 
\begin{align}
\E_{\beta_t \sim \mu} \log |1 - \beta_t^{-1} \lambda_t|
=
\log \frac{\sqrt{\kappa} - 1}{\sqrt{\kappa} + 1}\,.
\label{eq:equalization-unpredictability}
\end{align}
Intuitively, this can be interpreted as using random stepsizes to \emph{hedge} between the possible curvatures $\lambda_t$: the expected ``progress'' from an iteration is the same regardless of the curvature. Quantitatively, this equalization property~\eqref{eq:equalization-unpredictability} lets us argue an accelerated convergence rate using a martingale version of the Law of Large Numbers. In particular, this implies that the Arcsine distribution---which was optimal for the extremal problem~\eqref{eq:extremal} in the quadratic setting---achieves the same value for the extremal problem~\eqref{eq:extremal-convex} in the convex setting, and therefore is also optimal for this setting (as it is more general).

\begin{remark}[Overcoming unpredictability via random hedging]\label{rem:heging}
		Observe that in the extremal problem~\eqref{eq:extremal-convex}, we allow $\{\lambda_t\}$ to vary arbitrarily in $t$, and in particular we do not enforce these curvatures to be consistent with a single convex function. This consistency is the defining feature of Performance Estimation Problems~\citep{DT14,pesto} and is essential to the line of work on accelerating GD using deterministic stepsize schedules~\citep{altschuler2018greed,daccache2019performance,eloi2022worst,gupta22,grimmer23,alt23hedging1,alt23hedging2,GrimmerShuWang,GSW24,wang2024relaxed,zhang2024accelerated,grimmer2024composing,bok24}. Remarkably, \emph{random} stepsizes obtain the fully accelerated rate without needing to enforce any consistency conditions. This can be interpreted through the lens of ``hedging'' between worst-case functions, as described in~\citep{altschuler2018greed}. On one hand, the opportunity for faster convergence via \emph{deterministic} time-varying stepsizes arises due to the rigidity of convex functions: a function being worst-case for the stepsize in one iteration may constrain its curvature from being worst-case in the next iteration. On the other hand, the opportunity for faster convergence via \emph{random} stepsizes arises because the worst-case function does not ``know'' the realization of the random stepsizes. It is an interesting question if these two opportunities can be combined, e.g., can randomized analyses also exploit consistency conditions on $\{\lambda_t\}$ in order to generalize beyond separability? See the future work discussion in \S\ref{sec:future}.
\end{remark}

\section{Proof of main result}\label{sec:separable}

Here we prove Theorem~\ref{thm:sep:main}. We build up to the full result by first proving it for univariate objectives, and then extending to the general (separable) multivariate setting. See \S\ref{sec:hi} for a high-level overview of the analysis techniques and the two key conceptual challenges regarding unpredictability of the Hessian (dealt with in the univariate proof via the equalization property of the Arcsine distribution) and commutativity of the Hessian (dealt with in the multivariate proof via separability). 
For shorthand, we denote the accelerated rate by 
\begin{align}
    \Racc := \frac{\sqrt{\kappa} - 1}{\sqrt{\kappa} + 1
    }\,.
\end{align}

\subsection{Univariate setting}\label{ssec:sep:uni}

We begin by remarking that unforeseen complications can occur even in the univariate setting. For example, Polyak's Heavy Ball method fails to converge for such functions~\citep[\S4.6]{LRP16}.

We first state a simple helper lemma that lets us avoid assumptions of twice-differentiability when expanding the convergence rate of GD in the form~\eqref{eq:hi-convex:rate} described in the overview \S\ref{ssec:hi:convex}.

\begin{lemma}\label{lem:sep:uni:helper}
    Suppose that $f : \R \to \R$ is $m$-strongly convex and $M$-smooth, and let $x^*$ denote its minimizer. Then for any $x \neq x^*$, it holds that $$m \leq \frac{f'(x)}{x - x^*} \leq M\,.$$
\end{lemma}

For intuition, if $f$ is twice-differentiable, then this lemma follows immediately from the Intermediate Value Theorem since $f'(x) = f''(z) (x - x^*)$ for some $z$ between $x$ and $x^*$, and the Hessian $f''(z) \in [m,M]$ by $m$-strong convexity and $M$-smoothness. The proof is not much more difficult without the assumption of twice-differentiability.

\begin{proof}
	For the upper bound, $M$-smoothness of $f$ implies $|f'(x)| = |f'(x) - f'(x^*)| \leq M |x - x^*|$. Note that $\text{sign}(f'(x) - f'(x^*)) = \text{sign}(x - x^*)$ since $f$ is univariate. Thus, by re-arranging, we obtain the desired inequality $f'(x)  / (x - x^*) \leq M$. The lower bound follows from an analogous argument, since $m$-strong convexity of $f$ implies $|f'(x) - f'(x^*)| \geq m|x - x^*|$. 
\end{proof}

\begin{proof}[Proof of Theorem~\ref{thm:sep:main} for univariate $f$]
	Let $\beta_t$ denote the inverse stepsizes; these are i.i.d.\ $\ArcmM$. Defining $\lambda_t := \frac{f'(x_t)}{x_t - x^*}$, we have by definition of GD that
	\[
	x_{t+1} - x^* = x_t - x^* - \beta_t^{-1}  f'(x_t) = (1 - \beta_t^{-1} \lambda_t) (x_t - x^*).
	\]
	Therefore for any $n \in \mathbb{N}$, the average rate after $n$ steps is 
	\begin{align}
	R_n := \Bigg|\frac{x_{n} - x^*}{x_0 - x^*} \Bigg|^{1/n}
	=
	\Bigg|
	\prod_{t=0}^{n-1} (1 - \beta_t^{-1} \lambda_t)
	\Bigg|^{1/n}.
	\label{eq:rate-main-proof}
	\end{align}
	Note that this sequence of random variables is uniformly integrable. Indeed for each $n \in \mathbb{N}$, a crude bound gives $|
	\prod_{t=0}^{n-1} (1 - \beta_t^{-1} \lambda_t)|^{1/n}
	\leq \max_{\beta,\lambda \in [m,M]} |1 - \beta^{-1}\lambda|
	= \kappa - 1
	< \infty$.
	Therefore, by the Dominated Convergence Theorem, $L_1$ convergence of $R_n$ to $\Racc$ will follow immediately from almost-sure convergence. To show the latter, note that since almost-sure convergence is preserved under continuous mappings (in particular the $\exp(\cdot)$ function), it suffices to show that
	\begin{align}
	\frac{1}{n}
	\sum_{t=0}^{n-1} Z_t
	\asc
	\log  \Racc,
	\nonumber
	\end{align}
	where we define 
	\[
		Z_t := \log|1 - \beta_t^{-1} \lambda_t|\,.
	\]
    As mentioned in the overview section~\ref{ssec:hi:convex}, the issue with applying the Law of Large Numbers is that $\{Z_t\}$ are not i.i.d. (For example, even though we show below that these random variables have the desired first moments---which allows us to adapt the Law of Large Numbers---one can check that $\{Z_t\}$ are neither independent nor identically distributed.)
	\par Now observe that $\lambda_t$ depends on the realizations of only the \textit{previous} $\beta_0, \dots, \beta_{t-1}$. Moreover, $\lambda_t$ lies inside $[m, M]$ by Lemma~\ref{lem:sep:uni:helper}. Thus by the equalization property of the Arcsine distribution (\cref{lem:arc-equalization}), $\E[Z_t] = \log\Racc$ for all $t$, even if $\lambda_t$ is chosen adversarially as a function of the previous $\beta_0, \dots, \beta_{t-1}$. Indeed, by the independence of $\{\beta_t\}$ and the Law of Iterated Expectations:
	\begin{align}
		\E \Big[ Z_t \Big]
		&=
		\E_{\beta_0, \dots, \beta_t}
		\Big[ \log \abs{1 - \beta_t^{-1}\lambda_t} \Big] \nonumber
		\\ &=
		\E_{\beta_0, \dots, \beta_{t-1}} \Big[
		\E_{\beta_t} \Big[
		\log \abs{1 - \beta_t^{-1} \lambda_t} \; \Big| \; \beta_0, \dots, \beta_{t-1} \Big] \Big] \nonumber
		\\ &= \E_{\beta_0, \dots, \beta_{t-1}}
		\Big[ \log \Racc \Big] \nonumber
		\\ &= \log \Racc. \label{eq:sep:uni:main:expectation}
	\end{align}
	Therefore it suffices to show
	\begin{align}
		\frac{1}{n}
		\sum_{t=0}^{n-1} \left(Z_t - \E[Z_t]\right)
		\asc 0.
		\label{eq:goal-2}
	\end{align}
	We show this by adapting the standard proof of Kolmogorov's Strong Law of Large Numbers, see e.g.~\citep{martingale-book} for the classical proof. We first argue that
	$$
		M_n := \sum_{t=0}^{n-1} \frac{Z_t - \E[Z_t]}{t+1}
	$$
	forms a martingale w.r.t.$\,$the natural filtration $\cF_n = \sigma(\{\beta_t\}_{t < n})$, as stated in the following lemma. Intuitively this follows from the calculation~\eqref{eq:sep:uni:main:expectation}; a formal proof is deferred to Appendix~\ref{app-pf:lem:L1-martingale}. 
	\begin{lemma}[Martingale helper lemma]\label{lem:L1-martingale}
		$\{M_n\}$ is an $L_1$-bounded martingale w.r.t.\,$\{F_n\}$.
	\end{lemma}
	
	Thus we may apply the Martingale Convergence Theorem (see e.g.,~\citep[\S11.5]{martingale-book}) to establish that $M_{\infty} := \lim_{n \to \infty} M_n$ exists and is finite almost surely. Therefore by Kronecker's Lemma (recalled  below), we conclude that $\frac{1}{n}\sum_{t=0}^{n-1} (Z_t - \E[Z_t]) \asc 0$, proving~\eqref{eq:goal-2} as desired.
\end{proof}

\begin{lemma}[Kronecker's Lemma]
	Let $\{z_n\}$ and $\{b_n\}$ be two real-valued sequences such that each $b_n$ is strictly positive, the $b_n$ are strictly increasing to $\infty$, and $\sum_{t=0}^{\infty} \tfrac{z_t}{b_t}$ converges to a finite limit. Then $\lim_{n \to \infty} \tfrac{1}{b_n} \sum_{t=0}^{n-1} z_t = 0$.
\end{lemma}

\subsection{Multivariate setting}\label{ssec:sep:sep}
Here we prove Theorem~\ref{thm:sep:main} in the general multivariate setting. The proof follows quickly from the univariate setting proved above and the definition of separability.

\begin{proof}[Proof of Theorem~\ref{thm:sep:main}] 
	First, note that it suffices to prove the theorem for functions which are separable in the standard basis, i.e., functions $f$ of the form
	\begin{align}
		f(x) = \sum_{i=1}^d f_i([x]_i)\,.
		\label{eq:sep-std-basis}
	\end{align}
    Indeed, consider a general separable function $f(x) = \sum_{i=1}^d f_i([Ux]_i)$ where $U \in \R^{d \times d}$ is orthogonal. Then by the Chain Rule, $\nabla f = U^T \nabla h$ where $h(y) := \sum_{i=1}^d f_i([y]_i)$. Hence running GD $x_{t+1} = x_t - \alpha_t \nabla f(x_t)$ on $f$ is equivalent to running GD $y_{t+1} = y_t -\alpha_t \nabla h(y_t)$ on $h$, in the sense that $y_t = Ux_t$ for all $t$. Thus, modulo an isometric change of coordinates which does not affect the convergence rate, we may assume henceforth that the objective function is separable of the form~\eqref{eq:sep-std-basis}.

	\par Since $f(x) = \sum_{i=1}^d f_i([x]_i)$, the $i$-th coordinate of $\nabla f(x)$ is $f_i'([x]_i)$. Thus the GD update $
	x_{t+1} = x_t - \beta_t^{-1} \nabla f(x_t)
	$
	can be re-written coordinate-wise as
	\begin{align}
		[x_{t+1}]_i = [x_t]_i -  \beta_t^{-1}  f_i'([x_t]_i), \qquad \forall i \in [d]\,.
		\label{eq:sep:uni:main:update-coord}
	\end{align}
	Now we make two observations. First,~\eqref{eq:sep:uni:main:update-coord} is the update that GD would perform to minimize $f_i$ from initialization $[x_0]_i$. Second, each $f_i$ is $m$-strongly convex and $M$-smooth (because strong convexity and smoothness are preserved under restrictions). Thus we may apply the univariate setting of Theorem~\ref{thm:sep:main} (proved above) to argue that the asymptotic convergence rate is
	\[
	\limn
	\left(
	\frac{\abs{[x_{n}]_i - [x^*]_i}}{\abs{[x_0]_i -[x^*]_i}}
	\right)^{1/n}
	\ase
	\Racc
	\]
	for every coordinate $i \in S  := \{i \in [d] : [x_0]_i \neq [x^*]_i\}$ for which the initialization is not already optimal. Note that $x_0 \neq x^*$ ensures $S$ is non-empty. Since $[x^*]_i$ is the unique minimizer of $f_i$, and since also the number of coordinates $i \in S$ is finite, it follows that
	\[
	\limn \left(\frac{\|x_{n} - x^*\|}{\|x_0 - x^*\|} \right)^{1/n}
	\ase
	\limn
	\max_{i \in S}
	\left(
	\frac{\abs{[x_{n}]_i - [x^*]_i}}{\abs{[x_0]_i -[x^*]_i}}
	\right)^{1/n}
	\ase
	\Racc\,.
	\]
	This establishes the desired almost-sure convergence. $L_1$ convergence then follows from the Dominated Convergence Theorem, since we have uniform integrability by a similar argument as in the univariate setting above.
\end{proof}

\section{Generalizations}\label{sec:extensions}

The purpose of this paper is to show that random stepsizes can accelerate GD on classes of strongly convex and smooth functions that are more general than quadratics. For concreteness, Theorem~\ref{thm:sep:main} focuses on functions that are separable, strongly convex, and smooth. But this result extends to other classes of functions. We describe two extensions below, thematically splitting modifications of the separability assumption (\S\ref{ssec:extensions:separability}) from the strong convexity and smoothness assumptions (\S\ref{ssec:extensions:scs}). These two types of modifications are complementary in the sense that they individually deal with the two key issues discussed in \S\ref{ssec:hi:convex}: commutativity and unpredictability of the Hessian, respectively.

\begin{remark}[Performance metrics]
    We also mention that throughout, we measure convergence via the contraction factor $(\frac{\|x_n - x^*\|}{\|x_0 - x^*\|})^{1/n}$. The results are unchanged if the progress measure is changed from distance to function suboptimality or gradient norm, at either the initialization $x_0$, termination $x_n$, or both. This is because for $\kappa$-conditioned functions, all these progress measures $\|x - x^*\|^2$, $f(x) - f(x^*)$, and $\|\nabla f(x)\|^2$ are equivalent up to a multiplicative factor $c$ that is independent of $n$ (it depends only on the strong convexity and smoothness parameters), and therefore in the limit $n \to \infty$, the convergence rate is unaffected by changing these performance metrics since $\limn c^{1/n} = 1$.
\end{remark}

\subsection{Beyond separability}\label{ssec:extensions:separability}

Here we consider one example of a structural assumption that is different from separability yet still allows for the acceleration phenomenon via random stepsizes. To introduce this, recall from Definition~\ref{def:sep} that separability amounts to decomposability into a certain type of building block: univariate functions. Here we consider another building block: radial functions. 

\begin{defin}[Radial functions]
	A function $r : \R^d \to \R$ is \emph{radial} if $r(x) = h(\|x - x^*\|^2)$
	for some function $h : \R \to \R$ and point $x^* \in \R^d$.
\end{defin}

We begin by showing acceleration for the case of a single radial function, before moving to functions built from multiple radial functions. In words, the key idea behind the proof is that although radial functions do not have commutative Hessians---which we exploited in our main result to reduce the (multivariate) problem of convergence analysis to an analytically tractable (univariate) extremal problem, see the overview  \S\ref{ssec:hi:convex}---radial functions are univariate in nature (in particular all gradients in a GD trajectory are colinear), and this enables a similar reduction.

\begin{theorem}[Random stepsizes accelerate GD on radial functions]\label{thm:radial}
	Theorem~\ref{thm:sep:main} is unchanged if the assumption of separability is replaced with radiality. 
\end{theorem}
\begin{proof}
    	First, note that the center $x^*$ of a convex radial function is a minimizer. (Indeed, for any $x \in \R^d$, we have by radiality and then Jensen's inequality that $f(x) = \E_z f(z) \geq f(x^*)$, where the expectation is over $z$ drawn uniformly at random from the sphere of radius $\|x - x^*\|$ around $x^*$.)
	Thus after a possible translation, we may assume without loss of generality that $x^* = 0$, so that
	\[
	f(x) = h(\|x\|^2)
	\]
	for some univariate function $h : \R \to \R$. The chain rule implies $\nabla f(x) = 2h'(\|x\|^2) x$, hence the GD update $x_{t+1} = x_t - \alpha_t \nabla f(x_t)$ expands to
	\begin{align}
		x_n = \prod_{t=0}^{n-1} ( 1 - \alpha_t \lambda_t) x_0
		\qquad \text{ where } \qquad
		\lambda_t := 2h'(\|x_t\|^2)\,.
		\label{eq:radial-univariate}
	\end{align} 
	It is straightforward to show that $\lambda_t \in [m,M]$, e.g., this follows from Lemma~\ref{lem:sep:uni:helper} after restricting $f$ to the line between $x^*$ and $x_t$. Thus this
	gives an alternative approach for reducing the (multivariate) convergence rate to the same (univariate) product form~\eqref{eq:hi-convex:rate}, at which point the rest of our analysis goes through unchanged. 
\end{proof}

It follows that acceleration via random stepsizes generalizes to functions which can decompose into radial functions and univariate convex functions.

\begin{cor}
    Theorem~\ref{thm:sep:main} is unchanged if the assumption of separability is replaced with the following: suppose $f$ admits a 
    decomposition of the form
	\begin{align}
	f(x) = \sum_{i=1}^b f_i([Ux]_{S_i})\,,
	\label{eq:def-radially-separable}
\end{align} 
    for some orthogonal matrix $U \in \R^{d \times d}$, partition $S_1, \dots, S_b$ of $[d]$, and functions $f_1, \dots, f_b$ which are radial and/or univariate. (Above, $[Ux]_{S_i}$ denotes the subset of coordinates of $Ux$ in set $S_i$.)
\end{cor}
\begin{proof}
    Follows by the same argument as in \S\ref{ssec:sep:sep}, since we have established accelerated convergence for both types of blocks in this decomposition: univariate (\S\ref{ssec:sep:uni}) and radial functions (Theorem~\ref{thm:radial}).
\end{proof}

\subsection{Beyond strong convexity and smoothness}\label{ssec:extensions:scs}

Recall that for twice-differentiable functions $f$, the assumptions of $m$-strong convexity and $M$-smoothness are equivalent to the assumption that the spectrum of $\nabla^2 f(x)$ lies in the interval $[m,M]$ for all points $x$. The algorithmic techniques developed in this paper extend to functions~$f$ whose Hessians have more structured spectra, lying within some set $E \subset \R_{> 0}$ that is not necessarily an interval.
Note that strong convexity of $f$ ensures a strictly positive real spectrum. Such extensions to more structured Hessians have led to improved results for other algorithms in quadratic settings and beyond, see e.g.,~\citep{kelner2022big,oymak2021provable,goujaud2022super}.

\par This extension is rooted in connections to logarithmic potential theory, and proceeds as follows. Recall from the overview \S\ref{sec:hi} that the asymptotic convergence rate of GD with inverse stepsizes i.i.d.\ from a distribution $\mu$, is related to the maximum logarithmic potential of $\mu$ over $[m,M]$, i.e.,
\begin{align}
	\sup_{\lambda \in [m,M]} \E_{\beta \sim \mu} \log |1 - \beta^{-1} \lambda|\,,
\end{align}
and that the unique distribution minimizing this quantity is the Arcsine distribution (Lemma~\ref{lem:arc-extremal}).
This extends to spectral sets $E \subset \R_{> 0}$ beyond intervals, in particular the finite union of non-singleton compact intervals. Indeed, by a similar derivation, the convergence rate in this case is related to the analogous extremal problem
\begin{align}
	\sup_{\lambda \in E} \E_{\beta \sim \mu} \log |1 - \beta^{-1} \lambda|\,,
\end{align}
and the unique distribution minimizing this quantity is the unique optimal stepsize distribution.
What is this optimal distribution? This is a classical topic in potential theory and there is a well-developed machinery (typically based on computing conformal mappings, Green's functions, or Fekete sets) for determining this distribution explicitly in closed-form when possible (typically this only occurs for very simple sets) or otherwise approximating it; we refer the interested reader to the surveys and textbooks~\citep{ransford,saff2013logarithmic,saff10,6steps}.

\section{Discussion and related results}\label{sec:discussion}

\begin{figure}
    \centering
    \includegraphics[width=0.7\linewidth]{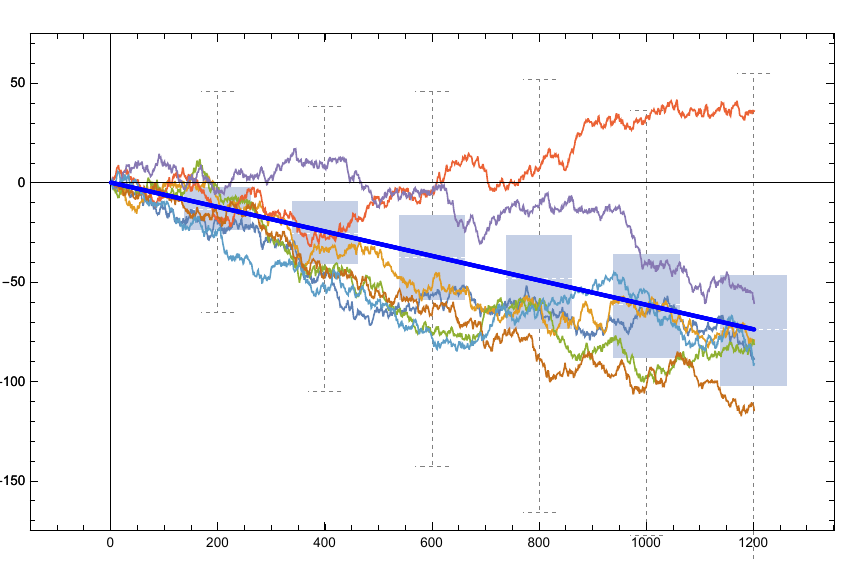}
    \caption{\footnotesize The randomness of the proposed stepsizes leads to variability between runs. Illustrated here with a boxplot for $500$ runs ($7$ full trajectories shown) from the same initialization on a univariate quadratic $f$, with $\kappa = 200$; similar phenomena occur for other problems. The horizontal axis is the number of iterations $n$; the vertical axis is $\log_{10} \tfrac{\|x_n - x^*\|}{\|x_0 - x^*\|}$. The worst run can diverge exponentially, but only with exponentially low probability (details in \S\ref{ssec:discussion:notions}). The median run converges exponentially at the optimal accelerated rate $\Racc$ (Theorem~\ref{thm:sep:main}). The fit is extremely precise, as shown by the bolded blue line. Notice that the best run can be substantially better than the median run.}
	\label{fig:runs}
\end{figure}

\subsection{Divergent runs and notion of convergence}\label{ssec:discussion:notions}

When considering randomized optimization algorithms, there are several notions for measuring the convergence rate. Two natural ones are:
\begin{align}
	\E \Bigg[ \Bigg( \frac{\|x_{n} - x^*\|}{\|x_0 - x^*\|} \Bigg)^{1/n} \Bigg] \qquad \text{ vs } \qquad 
	\E\Bigg[\frac{\|x_{n} - x^*\|}{\|x_0 - x^*\|} \Bigg]^{1/n} \,.
	\label{eq:notions}
\end{align}
These are different. Indeed, our main result shows that by running GD with random stepsizes, the former quantity is asymptotically equal to the accelerated rate $\Racc = \tfrac{\sqrt{\kappa} - 1}{\sqrt{\kappa} + 1} < 1$. Whereas the latter quantity can be larger than $1$ (details in Appendix~\ref{app:notions}). Why the discrepancy between these two notions of convergence rate? Which is relevant for optimization?

\par The discrepancy arises because the error $\|x_{n} - x^*\|$ is exponentially small with overwhelming probability (leading to our convergence result for the first notion), but can be exponentially large with exponentially small probability (leading to the instability for the second notion). \emph{The key point is that for optimization algorithms, it is very desirable to succeed with overwhelming probability, even if performance is poor in the unlikely complement event. Simply re-run the algorithm if that happens.} This is why we study the quantity on the left hand side of~\eqref{eq:notions}. For intuition, we conclude this section with an illustrative toy example that explains this discrepancy in its simplest form. 

\begin{example}[Measuring convergence when exponentially rare trajectories can be exponentially large]
	Consider i.i.d.\ random variables $Z_t$ which are each equally likely to be $1/9$ or $3$. Then $(\prod_{t=1}^n Z_t)^{1/n} = 3^{\tfrac{1}{n}\sum_{t=1}^n \log_3 Z_t} \to 3^{\E[\log_3 Z_t]} = 3^{-1/2} < 1$ by the Law of Large Numbers, hence
	\begin{align}
		\lim_{n \to \infty} \E\Big[ \Big( \prod_{t=1}^n Z_t \Big)^{1/n} \Big] = \frac{1}{\sqrt{3}} < 1\,.
		\label{ex-stability:1}
	\end{align}
	Whereas
	\begin{align}
		\lim_{n \to \infty} \E \Big[ \prod_{t=1}^n Z_t \Big]^{1/n} = \E[Z_1] = \frac{14}{9} > 1\,.
		\label{ex-stability:2}
	\end{align}
	In words, this discrepancy arises because
	$\prod_{t=1}^n Z_t$ can be exponentially large in $n$, but the probability of this is exponentially small. For example, $Z_1 = \dots = Z_n = 3$ occurs with probability $2^{-n}$ and thus by itself, this event contributes $\lim_{n \to \infty} (3^n 2^{-n})^{1/n} = 3/2 > 1$ to the expectation in~\eqref{ex-stability:2}. Whereas the $1/n$ normalization inside the expectation in~\eqref{ex-stability:1} dampens the effect of exponentially rare blow-ups, and thus agrees with the exponentially fast convergence that occurs with overwhelming probability.
\end{example}

This distinction is analogous to the classical Kelly criterion in information theory and gambling: for repeated multiplicative bets, the objective governing long-run typical performance is not the arithmetic expectation of terminal wealth, but rather the expected logarithmic growth rate, or equivalently the asymptotic geometric mean \citep{Kelly1956,Breiman1961,CoverThomas2006}. Here the sign is reversed---we seek fast multiplicative decay of error rather than fast growth of wealth---but the principle is the same. The normalized logarithm of the error captures the almost-sure exponential rate, whereas the expectation of the unnormalized error can be controlled by exponentially rare trajectories.

\subsection{Typical runs and high probability bounds}\label{ssec:discussion:hp}

Theorem~\ref{thm:sep:main} establishes that random stepsizes lead to an accelerated rate in the limit as the number of iterations $n \to \infty$. The same analysis arguments can also establish non-asymptotic convergence rates that apply for any finite $n$ and hold with high probability. Indeed, recall from~\eqref{eq:rate-main-proof} that our argument in \S\ref{sec:separable} relates the convergence rate (for any finite $n$) to an empirical sum, namely
\begin{align}
	\Bigg|\frac{x_{n} - x^*}{x_0 - x^*} \Bigg|^{1/n}
	=
	\exp\left( \frac{1}{n} \sum_{t=0}^{n-1} \log |1 - \beta_t^{-1} \lambda_t| \right)\,.
	\label{eq:rate-hp}
\end{align}
Now, instead of using the Law of Large Numbers to argue that this empirical sum converges to its expectation (giving the accelerated rate), one can use concentration-of-measure to argue about the deviations from its expectation.

Below we state such a result. This is established by bounding the variance of the aforementioned empirical sum via a martingale-type argument and applying Chebyshev's inequality; and then extending this convergence rate to general multivariate objectives $f$ by a union bound. Full proof details are deferred to Appendix~\ref{app:hp}. We note that one can show more refined bounds by bounding higher-order moments or the moment generating function (see, e.g., the textbooks~\citep{van2014probability,vershynin2018high,dembo2009large}), but that requires computing involved logarithmic integrals and thus for simplicity of exposition, here we state only this simple illustrative result.

\begin{theorem}[Non-asymptotic version of Theorem~\ref{thm:sep:main}]\label{thm:sep:hp}
	Consider the setup of Theorem~\ref{thm:sep:main}. For any deviation $\eps > 0$ and error probability $\delta > 0$, 
	\begin{align}
		\Prob\left( \left(\frac{\|x_{n} - x^*\|}{\|x_0 - x^*\|} \right)^{1/n}  \geq e^{\eps} \Racc\right) \leq \delta
	\end{align}
	if $n \geq \left( \pi^2 + o_{\kappa}(1) \right) \frac{d}{\delta \eps^2}$.
\end{theorem}

This implies that for a typical run, the convergence rate is close to the accelerated rate; see Figure~\ref{fig:runs} for an illustration.

\subsection{Stability}\label{ssec:discussion:inexact}

Throughout, we have assumed for simplicity that all calculations are performed with exact gradients and in exact arithmetic. It is of practical importance to understand how inexactness affects convergence, and there is a long line of work on such questions for first-order algorithms, see, e.g.,~\citep{de2020worst, polyak1971convergence,devolder2014first,d2008smooth,schmidt2011convergence,LebedevFinogenov71,AgarwalGoelZhang} and the references within. There are many models of inexact computation. Here, as an illustrative example, we study one common model: the \emph{multiplicative error model for gradients}~\citep{de2020worst, polyak1971convergence}. Specifically, suppose that all calculations are still performed in exact arithmetic, but GD performs updates using an estimate $\tilde{g}_t$ of the true gradient $ \nabla f(x_t)$, where
\begin{align}
	\|\tilde{g}_t - \nabla f(x_t)\|  \leq \eps \|\nabla f(x_t)\|\,,
	\label{eq:inexact}
\end{align}
for some inexactness parameter $\eps \in (0,1)$. This model of inexact gradients does not align with separability, hence we state the following results for univariate objectives for simplicity. It is an interesting question to extend beyond.

\begin{theorem}[Inexact-gradient version of Theorem~\ref{thm:sep:main}]\label{thm:inexact}
	Consider the setup of Theorem~\ref{thm:sep:main} for $f$ univariate.  Suppose that GD updates in each iteration using $\tilde{g}_t$ rather than $\nabla f(x_t)$, i.e., 
	\begin{align}
		x_{t+1} = x_t - \alpha_t \tilde{g}_t\,,
	\end{align} 
	where $\tilde{g}_t$ are inexact gradients satisfying the multiplicative-error assumption~\eqref{eq:inexact}. Then the convergence rate of GD with i.i.d.\ Arcsine stepsizes is almost surely bounded by
	\begin{align}
		\limn \left( \frac{\|x_n - x^*\|}{\|x_0 - x^*\|} \right)^{1/n} \leq \;\;
		\underbrace{\frac{1}{1 + \tilde{\eps} - \sqrt{(1 + \tilde{\eps})^2 - 1}}}_{\text{slowdown from inexact gradients}}
		\cdot
		\underbrace{\frac{\sqrt{\kappa} - 1}{\sqrt{\kappa} + 1}}_{\text{accelerated rate }\Racc}
		\, \text{ where } \quad \tilde{\eps} := \frac{2M}{M-m}\,\eps
		\,.
		\label{eq:thm-inexact}
	\end{align}
	Moreover this bound is tight, in that it holds with equality for some choice of $f$ and inexact gradients. 
\end{theorem}

The proof is deferred to Appendix~\ref{app:inexact}, since it is similar to that of Theorem~\ref{thm:sep:main}. Conceptually, the only difference is that the average Hessians $\lambda_t \in [m,M]$ are now perturbed multiplicatively, and thus the effective curvatures $\tilde{\lambda}_t$ lie in the enlarged interval $[(1-\eps)m,(1+\eps)M]$. The convergence rate is therefore governed by $\exp( \E_{\beta \sim \ArcmM} \log | 1- \beta^{-1} \tilde{\lambda}|)$ where $\tilde{\lambda}$ can now be slightly outside the interval $[m,M]$. When the inverse stepsizes are still drawn from $\ArcmM$, the worst perturbation occurs at the upper endpoint $(1+\eps)M$, which gives the bound in Theorem~\ref{thm:inexact}.

We remark that for small error $\eps$, the slowdown factor in Theorem~\ref{thm:inexact} is of the order
\begin{align}
	\frac{1}{1 + \tilde{\eps} - \sqrt{(1 + \tilde{\eps})^2 - 1}}
	= 1 + \sqrt{2 \tilde{\eps}} + \Theta(\tilde{\eps})
	= 1 + \sqrt{\frac{4M\eps}{M-m}} + \Theta(\eps)\,.
    \label{eq:slowdown-1}
\end{align}
In other words, for large condition number $\kappa \gg 1$, the number of iterations required by GD increases by a multiplicative factor of roughly $1 + 2\sqrt{\eps}$ when using inexact gradients. Note that this slowdown is more than the $(1 + \Theta(\eps))$ slowdown for constant-stepsize schedules~\citep[Theorem 5.3]{de2020worst}. This is a manifestation of the general phenomenon that accelerated algorithms are often less stable to model mis-specification, see e.g.,~\citep{devolder2014first}.

The above discussion analyzes the effect of inexact gradients for the i.i.d.\ stepsize schedule that is optimized for exact gradients; one can also ask for the i.i.d.\ stepsize schedule that is optimized for this setting of inexact gradients. This follows from a minor adaptation of our techniques. 

\begin{theorem}[Optimized i.i.d.\ stepsize schedules for inexact gradients]\label{thm:inexact-reoptimized}
	Consider the setup of Theorem~\ref{thm:inexact}. Let $\tilde{m} := (1-\eps)m$, $\tilde{M} := (1 + \eps)M$, and $\tilde{\kappa} := \tfrac{\tilde{M}}{\tilde{m}} = \tfrac{1+\eps}{1-\eps} \kappa$. By using i.i.d.\ inverse stepsizes $\alpha_t^{-1}$ from $\Arc{\tilde{m}}{\tilde{M}}$, GD converges at a rate that is almost surely bounded by
	\begin{align}
		\limn \left( \frac{\|x_n - x^*\|}{\|x_0 - x^*\|} \right)^{1/n} \leq 
        \frac{\sqrt{\tilde{\kappa}} - 1}{\sqrt{\tilde{\kappa}} + 1}
		\,.
		\label{eq:thm-inexact-reoptimized}
	\end{align}
    Moreover, this bound is tight in that it holds with equality for some choice of $f$ and inexact gradients, and also this is the unique stepsize distribution for which GD achieves this rate.
\end{theorem}

The key idea is that the inexact gradients effectively worsen the parameters for strong convexity (from $m$ to $\tilde{m}$) and smoothness (from $M$ to $\tilde{M}$). Theorem~\ref{thm:inexact-reoptimized} then follows by appealing to our main result, Theorem~\ref{thm:sep:main}, for this enlarged range. Proof details are deferred to Appendix~\ref{app:inexact}. By a Taylor expansion in $\eps$, the resulting convergence rate $\tilde{R}_{\mathrm{acc}} :=  \tfrac{\sqrt{\tilde{\kappa}}-1}{\sqrt{\tilde{\kappa}}+1}$ has a slowdown factor of
\begin{align}
    \frac{\tilde{R}_{\mathrm{acc}}}{\Racc}
    =
    1 + \frac{2\sqrt{\kappa}}{\kappa-1}\eps + O(\eps^2)
    =
    1 + \frac{2\sqrt{Mm}}{M-m}\eps + O(\eps^2)\,.
    \label{eq:slowdown-2}
\end{align}
This $1+\Theta(\eps)$ slowdown achieved by re-optimizing the stepsizes improves over the $1+\Theta(\sqrt{\eps})$ slowdown for using the original $\ArcmM$ stepsizes, c.f.,~\eqref{eq:slowdown-1}.

\subsection{Game-theoretic connections and lower bounds}\label{ssec:discussion:lb}

This paper considers random stepsizes for the algorithmic purpose of accelerating GD. It turns out that this randomized perspective is also helpful for the dual problem of constructing algorithmic lower bounds. We describe this duality here. 

\par Recall from \S\ref{sec:hi} that the optimal distribution $\mu$ over i.i.d.\ inverse stepsizes $\beta = \alpha^{-1}$ has the following extremal characterization:
\begin{align}
	\inf_{\mu \in \cM(E)} \; \sup_{\lambda \in E}\;  \E_{\beta \sim \mu} \log \abs{1 - \beta^{-1}\lambda}\,.
	\label{eq:games:1}
\end{align}
Above, $\mu$ is constrained to $E := [m,M]$ without loss of generality (since this holds at optimality), $\lambda \in E$ should be interpreted as the curvature of a ``worst-case'' univariate quadratic function $f(x) = \tfrac{\lambda}{2}x^2$, and $\E_{\beta \sim \nu} \log |1 - \beta^{-1} \lambda|$ is the logarithm of the asymptotic convergence rate.

\par To bridge this problem of accelerating GD with the problem of constructing lower bounds, we consider a standard two-player zero-sum game where players choose their actions simultaneously and independently.
In our case, one player chooses inverse stepsizes $\beta^{-1}$ and the other player chooses a quadratic function (parameterized by its curvature $\lambda)$. By lifting to mixed (i.e., randomized) strategies, one player chooses a distribution $\mu$ over stepsizes, the other chooses a distribution $\nu$ over quadratic functions, with payoff given by ${\cal R}(\mu,\nu) := 
\E_{\beta \sim \mu, \, \lambda \sim \nu} \log \abs{1 - \beta^{-1}\lambda}$. 

\par The key observation is that due to the specific payoff function, the game is essentially symmetric between the two players (modulo inversion), and as a consequence, so are the optimal mixed strategies for both players. Below, let $\FlippedArcmM$ denote the law of $\lambda$ when $\lambda^{-1} \sim \Arc{M^{-1}}{m^{-1}}$. 

\begin{lemma}[Optimal distribution for worst-case functions]\label{lem:flipped-arc}
    The (minimax) value of this game is $\log \Racc$, i.e., we have the saddle point condition
\[
\sup_\nu {\cal R}(\mu^*,\nu)
=
{\cal R}(\mu^*,\nu^*)
=
\inf_\mu {\cal R}(\mu,\nu^*)
\]
with
\[
{\cal R}(\mu^*,\nu^*) = 
\E_{\beta \sim \mu^*, \, \lambda \sim \nu^*} \log \abs{1 - \beta^{-1}\lambda} =
\log \Racc,
\]
and the unique optimal mixed strategies are $\mu^* = \ArcmM$ and $\nu^* = \FlippedArcmM$. 

\begin{remark}[Game-theoretic interpretation of equalization]
    Notice that both optimal distributions have full support. As a consequence,  at  equilibrium a player is indifferent among all of their pure actions, i.e., the equalization property (Lemma~\ref{lem:arc-equalization}) must hold. 
\end{remark}

\end{lemma}

\begin{proof}[Proof sketch for Lemma~\ref{lem:flipped-arc}]
    Lemma~\ref{lem:arc-extremal} implies that $\mu^* = \ArcmM$ is the unique optimal solution for the stepsize player, and that the value of the game is $\log \Racc$. See~\citep[Chapter 3.3.3]{altschuler2018greed} for a rigorous proof that the unique optimal solution for the other player is $\nu^* = \FlippedArcmM$; here we just provide intuition in terms of symmetry of the game. The underlying idea is to consider a ``flipped game'' in which strategies correspond to the inverses of strategies in the original game, i.e., one player chooses stepsizes $\tilde{\beta} = \beta^{-1} \in \tilde{E}$, and the other player chooses a quadratic function parameterized by its inverse curvature $\tilde{\lambda} = \lambda^{-1} \in \tilde{E}$. Here $\tilde{E} := [\tilde{m},\tilde{M}]$ where $\tilde{m} := M^{-1}$ and $\tilde{M} := m^{-1}$. By lifting to mixed strategies, one player chooses a distribution $\tilde{\mu}$ over stepsizes $\tilde{\beta} \in \tilde{E}$, and the other player chooses a distribution $\tilde{\nu}$ over quadratic functions $\tilde{\lambda} \in \tilde{E}$, with payoff given by 
    \[
        {\cal \tilde{R}}(\tilde{\mu},\tilde{\nu}) :=  \E_{\tilde{\beta} \sim \tilde{\mu}, \tilde{\lambda} \sim \tilde{\nu}} \log |1 - \tilde{\lambda}^{-1} \tilde{\beta}|\,.
    \]
    This flipped game is equivalent to the original game due the invariance of both the payoff function $\tilde{\cR}(\tilde{\mu},\tilde{\nu}) = \cR(\mu,\nu)$ and the condition number $\tfrac{\tilde{M}}{\tilde{m}} =\tfrac{M}{m}$ with respect to this inversion of the strategies. Because $\ArcmM$ is the unique optimal solution $\mu^*$ in the original game, $\FlippedArcmMt$ is the unique optimal solution $\tilde{\mu}^*$ in the flipped game; in order to show that $\FlippedArcmM$ is the unique optimal solution $\nu^*$ in the original game, it suffices to show that $\ArcmMt$ is the unique optimal solution $\tilde{\nu}^*$ in the flipped game. To this end, it suffices to show that an optimal solution $\tilde{\nu}^*$ must satisfy the equalization property in the flipped game, i.e., 
    \begin{align}
        \cR(\delta_{\tilde{\beta}}, \tilde{\nu}^*) = \E_{\tilde{\lambda} \sim \tilde{\nu}^*} \log \abs{1 - \tilde{\lambda}^{-1} \tilde{\beta}}
        \nonumber
    \end{align}
    is independent of $\tilde{\beta} \in \tilde{E}$. This suffices since the equalization property uniquely defines the Arcsine distribution by classical results from logarithmic potential theory, see Appendix~\ref{app:pot}. 

    \par Why must an optimal solution $\tilde{\nu}^*$ for the flipped game satisfy this equalization property? An informal argument is as follows. By definition of $\tilde{\cR}$, the symmetry between $\tilde{\cR}$ and $\cR$, and then the equalization property of the Arcsine distribution (Lemma~\ref{lem:arc-equalization}) applied to $\mu^*$, 
    \begin{align*}
        \E_{\tilde{\beta} \sim \tilde{\mu}} \cR(\delta_{\tilde{\beta}}, \tilde{\nu}^*)
        =
        \tilde{\cR}(\tilde{\mu}^*, \tilde{\nu}^*)
        =
        \cR(\mu^*, \nu^*)
        =
        \log \Racc\,.
    \end{align*}  
    Thus, forgoing measure-theoretic technicalities, if $\tilde{\nu}^*$ did not satisfy the equalization property, then by the probabilistic method, there would exist some $\tilde{\beta}$ for which $\tilde{\beta} \mapsto \cR(\delta_{\tilde{\beta}}, \tilde{\nu}^*)$ is strictly less than its average $\E_{\tilde{\beta} \sim \tilde{\mu}} \cR(\delta_{\tilde{\beta}}, \tilde{\nu}^*) = \log \Racc$. But this means that the value of the flipped game is less than $\log \Racc$, hence also for the original game (since their values are identical), contradiction.
\end{proof}

As a corollary, Lemma~\ref{lem:flipped-arc} gives a new proof that $\tfrac{\sqrt{\kappa} - 1}{\sqrt{\kappa} + 1}$ is the fastest possible convergence rate for GD. 
This result is classically due to~\citep{nem-yudin}, who proved this by constructing a single, infinite-dimensional ``worst-case function'' for which any Krylov-subspace algorithm converges slowly. Here we construct a ``worst-case distribution'' over univariate functions for which any GD stepsize schedule converges slowly in expectation. 
By the probabilistic method, for any specific GD stepsize schedule, we can then conclude the existence of a single univariate function for which convergence is slow. For simplicity, we state this lower bound here for i.i.d.\ stepsizes; this can be extended to arbitrary GD stepsize schedules and also to non-asymptotic lower bounds, see~\citep[Chapter 5]{altschuler2018greed}.

\begin{cor}\label{cor:lb}
	For any distribution of i.i.d.\ GD stepsizes, there exists a univariate quadratic function that is $m$-strongly convex and $M$-smooth such that the asymptotic convergence rate
	\begin{align*}
		\limn \E \left[ \left( \frac{\|x_n - x^*\|}{\|x_0 - x^*\|} \right)^{1/n} \right] \geq \frac{\sqrt{\kappa} - 1}{\sqrt{\kappa} + 1}\,.
	\end{align*}
\end{cor}
\begin{proof}
	Lemma~\ref{lem:flipped-arc} shows that in expectation over $\lambda \sim \FlippedArcmM$, the asymptotic convergence rate is $\Racc$ when run on the quadratic function $f_{\lambda}(x) = \tfrac{\lambda}{2} x^2$ from initialization $x_0 = 1$. By the probabilistic method, there exists a $\lambda^*$ for which the asymptotic convergence rate on $f_{\lambda^*}(x)$ is at least this average value $\Racc$.
\end{proof}

\begin{remark}
	The lower bounds in this section can be interpreted as an instantiation of Yao's minimax principle~\citep{yao}: in order to lower bound the performance of a randomized algorithm (e.g., GD with the Arcsine Stepsize Schedule), we exhibit a  distribution over difficult inputs (namely quadratics with curvature $\lambda$ according to the FlippedArcsine distribution) for which any deterministic algorithm cannot perform well. This is implicit along the way in this section's development. 
\end{remark}

\section{Conclusion and future work}\label{sec:future}

This paper shows that randomizing stepsizes provides an alternative mechanism for accelerated convex optimization beyond the special case of quadratics. This raises several natural questions about the possible benefit of randomness for GD---or said equivalently, whether the Arcsine Stepsize Schedule can be derandomized. The current results differ on two fronts: 
Theorem~\ref{thm:sep:main} shows that random stepsizes achieve the fully accelerated rate $O(\kappa^{1/2} \log \tfrac{1}{\eps})$ on separable functions, whereas~\citep{alt23hedging1} shows that deterministic stepsize achieves the partially accelerated ``Silver'' rate $O(\kappa^{\log_{1+\sqrt{2}} 2} \log \tfrac{1}{\eps}) \approx O(\kappa^{0.78} \log \tfrac{1}{\eps})$  on general convex functions. These two sources of differences motivate the questions:

\begin{itemize}
    \item Is it possible to de-randomize Theorem~\ref{thm:sep:main}, i.e., construct a deterministic stepsize schedule which achieves the fully accelerated rate for separable functions? If so, a natural conjecture would be some ordering of the Chebyshev stepsize schedule. 
    \item Does the Arcsine Stepsize Schedule fully accelerate on general (non-separable) convex functions? If not, can any randomized stepsize schedule improve over the Silver Convergence Rate for general convex functions?
\end{itemize}

There are also many questions about extensions to other optimization settings of interest. 
For other settings, what is the optimal stepsize distribution if not Arcsine? In any setting, is there a provable benefit of using non-i.i.d.\ stepsizes? Can random stepsizes be alternatively understood through discretization analyses of continuous-time gradient flows, or via the continuized framework of~\citep{continuized}? While other mechanisms for acceleration (e.g., momentum) are now well understood~\citep{d2021acceleration}, little is known about acceleration via randomness.

    \paragraph*{Acknowledgements.} 
    We are grateful to Abraham Wyner for suggesting connections to repeated gambling, discussed in \S\ref{ssec:discussion:notions}.

	\newpage
	\appendix

\section{Connections to logarithmic potential theory}\label{app:pot}

Recall from the overview section \S\ref{sec:hi} that the optimal distribution $\mu$ for inverse stepsizes $\beta$ is given by the extremal problem
\begin{align}
	\inf_{\mu \in \cM(E)} \;\sup_{\lambda \in E} \;\E_{\beta \sim \mu} \log \abs{1 - \beta^{-1} \lambda }\,,
	\label{eq:extremal-app}
\end{align}
where $E := [m,M]$ and $\cM(E)$ denotes the probability distributions supported on $E$.\footnote{Although the stepsize problem was not originally written with this constraint in \S\ref{sec:hi}, it is easily seen that this constraint is satisfied at optimality and therefore does not change the extremal problem. See Observation~\ref{obs:constrained}.} Indeed, for quadratic optimization, this extremal problem exactly characterizes the asymptotic convergence rate (see \S\ref{ssec:hi:quad}); and for the more general setting of separable convex optimization, even though the relevant extremal problem is somewhat different, the solution is identical (see \S\ref{ssec:hi:convex}). 

In this appendix section, we connect this extremal problem over stepsize distributions with the extremal problem for electrostatic capacitors that is classically studied in potential theory. 
As we explain below, although these two problems seek probability distributions minimizing different objectives, the two problems are intimately linked because they have the same optimality condition: the optimal distribution $\mu$ must satisfy the equalization property from Lemma~\ref{lem:arc-equalization} (i.e., the objective $\E_{\beta \sim \mu} \log |1 - \beta^{-1}\lambda|$ has equal value for all $\lambda \in [m,M]$). Because this equalization property uniquely characterizes the distribution (a classical fact from potential theory, see Lemma~\ref{lem:pt:frostman-converse}), this means that the two problems have the same unique optimal solution. In this way, the equalization property plays the central role connecting these two problems.

This appendix section is organized as follows. In \S\ref{app:pot-electrostatics} we recall the electrostatics problem and highlight the syntactic similarities and differences to the stepsize problem. In \S\ref{app:pot-equalization} we explain the equalization property in the context of both problems and how it arises as the optimality condition. This is the key link between the two problems. In \S\ref{app:pot-proofs} we exploit this connection concretely by using the machinery of potential theory to solve the stepsize problem. This provides alternative proofs of Lemmas~\ref{lem:arc-extremal} and~\ref{lem:arc-equalization} that are significantly shorter than the integral calculations of~\citep{kalousek,pronzato11}.

\subsection{Extremal problem for electrostatic capacitors}\label{app:pot-electrostatics}

The extremal problem central to logarithmic potential theory is
\begin{align}
	\inf_{\mu \in \cM(E)} \iint \log \frac{1}{\abs{\lambda - \omega}} d\mu(\omega) d\mu(\lambda)\,,
	\label{eq:extremal-pot}
\end{align}
given a set $E \subset \C$.
This problem is sometimes called the electrostatics problem for a capacitor due to its physical interpretation: how should you place a distribution of repelling particles on $E$ so as to minimize the total repulsive energy?  Here the repulsive energy is measured via $$I(\mu) := \iint \log \frac{1}{\abs{\lambda - \omega}} d\mu(\omega) d\mu(\lambda) =  \E_{\lambda \sim \mu} \phi_{\mu}(\lambda)\,,$$ which is the $\mu$-average of the logarithmic potential energy at $\lambda$, denoted $$\phi_{\mu}(\lambda) := \E_{\omega \sim \mu} \log \frac{1}{|\lambda - \omega|}\,.$$

Standard arguments based on lower semi-continuity and strict convexity of the objective ensure that an optimal distribution exists and is unique. For certain sets $E$, this optimal distribution can be computed in closed form. 
For example it is well-known that for the interval $E = [m,M]$, the unique optimal distribution is the $\ArcmM$ distribution, see e.g.,~\citep[Example 1.11]{saff10}.

\paragraph*{Syntactic similarities and differences.} 
Let $E = [m,M]$. Then the stepsize extremal problem~\eqref{eq:extremal-app} clearly bears similarities to the electrostatics extremal problem~\eqref{eq:extremal-pot}, namely
\begin{align}
	\inf_{\mu \in \cM(E)} \sup_{\lambda \in E} \phi_{\mu}(0) - \phi_{\mu}(\lambda)
	\qquad \text{versus} \qquad
	\inf_{\mu \in \mathcal{M}(E)} \E_{\lambda\sim \mu} \phi_{\mu}(\lambda)\,.
\end{align}
Yet there are two differences. First, the stepsize problem evaluates the logarithmic potential $\phi_{\mu}(\lambda)$ at a specific point $\lambda \in [m,M]$, whereas the electrostatics problem averages this potential over $\lambda \sim \mu$. Second, the stepsize problem has an additional term of $\phi_{\mu}(0)$.

In general, changing the objective or constraints in an extremal problem can significantly alter the optimal solution. Yet, remarkably, the optimal solution $\mu$ is identical for both problems. This is because both problems have the same optimality condition, described next.

\subsection{Equalization property}\label{app:pot-equalization}

Here we describe how the equalization property arises as the optimality condition in both problems. 

\subsubsection{Equalized potential for the electrostatics problem}

A celebrated fact---sometimes called the Fundamental Theorem of Potential Theory~\citep[Chapter 3.3]{ransford}---is that the optimal distribution $\mu$ to the electrostatics problem~\eqref{eq:extremal-pot} satisfies an \emph{equalization property}: the potential $\lambda \mapsto \phi_{\mu}(\lambda)$ is constant on $E$. (Note that this is equivalent to the constancy on $E$ of the objective for the stepsize problem, namely $\lambda \mapsto \E_{\beta} \log |1 - \beta^{-1}\lambda| = \phi_{\mu}(0) - \phi_{\mu}(\lambda)$.) 
\par This equalization property has an intuitive physical interpretation: if the potential is not constant on $E$, then particles will flow to points of $E$ with lower potential, so the current state is not at equilibrium and therefore cannot be optimal. This is why the optimal solution $\mu$ is often called the equilibrium distribution. This result is attributed to Frostman; here we state a simple version of it that suffices for our setting, namely a compact interval $E = [m,M]$, although it extends to much more general sets $E \subset \C$~\citep[Chapter 3.3]{ransford}.

\begin{theorem}[Frostman's Theorem]\label{thm:pt:frostman}
	Let $E = [m,M]$. Then there is a unique optimal distribution $\mu^*$ for~\eqref{eq:extremal-pot}. Moreover it satisfies
	\begin{itemize}
		\item $\phi_{\mu^*}(\lambda) = I(\mu^*)$ for all $\lambda \in E$.
		\item $\phi_{\mu^*}(\lambda) \leq I(\mu^*)$ for all $\lambda \in \mathbb{C}$.
	\end{itemize}
\end{theorem}

A proof can be found in any textbook on logarithmic potential theory, e.g.,~\citep{tsuji,ransford,saff2013logarithmic}. Below we use calculus of variations to give an informal back-of-the-envelope calculation that provides intuition for how this equalization property arises as the optimality condition. Let 
\begin{align*}
	F(\mu) := \iint \log \frac{1}{|x-y|} d\mu(x) d\mu(y) + \gamma \left( \int d\mu(x) - 1\right)
\end{align*}
denote the original objective plus a Lagrangian term for the normalization constraint on $\mu$. Then a quick calculation shows that the first variation is 
\begin{align*}
	\delta F(\mu,h) &= \frac{d}{d\eps} F(\mu+ \eps h) \big|_{\eps = 0} 
	=
	\int \left[ \gamma + 2 \int \log \frac{1}{|x-y|} d\mu(x) \right] dh(y)\,.
\end{align*}
In order for this first variation to vanish along any direction $h$, it must hold that 
\begin{align*}
	\gamma + 2 \int \log \frac{1}{|x-y|} d\mu(x) = 0, \qquad \forall y \in E\,.
\end{align*}
This is the equalization property.

\subsubsection{Equalized convergence rate for the stepsize problem}\label{app:pot-equalization-stepsize}

As discussed in \S\ref{ssec:hi:quad} and formally stated in Lemma~\ref{lem:arc-extremal}, the stepsize problem has the same equalization property at optimality. This is rigorously proven in \S\ref{app:arc-extremal} using the machinery of potential theory; to provide intuition for how this arises as the optimality condition, here we give an informal back-of-the-envelope argument based on the game-theoretic perspective developed in \S\ref{ssec:discussion:lb}.

Suppose $\mu$ is a solution to the stepsize problem~\eqref{eq:extremal-app}. By the calculation in \S\ref{ssec:discussion:lb}, there is a distribution $\nu$ (the Flipped Arcsine distribution) for which 
\begin{align}
    \E_{\beta \sim \mu,\, \lambda \sim \nu}  \log \abs{1 - \beta^{-1}\lambda} = \log \Racc\,.
\end{align}
Observe that this is the $\nu$-average of $P_{\mu}(\lambda) := \E_{\beta \sim \mu} \log |1 - \beta^{-1}\lambda| = \phi_{\mu}(0) - \phi_{\mu}(\lambda)$. Thus, if $\mu$ does not equalize, then (forgoing measure-theoretic technicalities) there exists $\lambda \in \supp(\nu) = [m,M]$ for which $P_{\mu}(\lambda)$
is strictly larger than the $\nu$-average. That is, $\max_{\lambda \in [m,M]} P_{\mu}(\lambda) > \log \Racc$. But this contradicts the optimality of $\mu$ since the value of the stepsize problem~\eqref{eq:extremal-app} is $\log \Racc$. Therefore the equalization property must be satisfied by the optimal distribution $\mu$.

\subsection{Solution to the extremal problem over stepsizes}\label{app:pot-proofs}

So far we have focused on developing conceptual connections between the stepsize problem and electrostatics problem. We now exploit these connections concretely by using the machinery of potential theory to solve the stepsize problem. This provides short proofs of Lemmas~\ref{lem:arc-extremal} and~\ref{lem:arc-equalization}.

\subsubsection{Proof of Lemma~\ref{lem:arc-equalization}}\label{app:arc-equalization}

Here we prove the equalization property.
In the optimization literature, this was proven by computing the integral directly using trigonometric identities~\citep{pronzato11}, or by decomposing the logarithmic integrand into the Chebyshev polynomial basis~\citep{kalousek}. Here we mention an alternative approach by appealing to the following well-known identity from potential theory, which gives the potential for the equilibrium distribution on $[-1,1]$ at any point $z$; see e.g.,~\citep[Example 1.11]{saff2013logarithmic}.

\begin{fact}\label{fact:eq}
	Let $\gamma$ denote the Arcsine distribution on $[-1,1]$. Then for any $z \in \R$,
	\begin{align*}
		\E_{t \sim \gamma} \log \frac{1}{\abs{z-t}} = \log 2 - \log \left( \abs{z} + \sqrt{z^2-1}\right) \cdot \mathds{1}_{z \notin [-1,1]}\,.
	\end{align*}
\end{fact}

\begin{proof}[Proof of Lemma~\ref{lem:arc-equalization}]
	Let $\mu$ denote the Arcsine distribution on $[m,M]$. Note that $\mu$ is the push-forward of $\gamma$ under the linear map $L(z) := \tfrac{M+m}{2} + \tfrac{M-m}{2}z$ that sends $[-1,1]$ to $[m,M]$. Thus by a change-of-measure, the $\mu$-potential at any $\omega \in \C$ is
	\begin{align*}
		\phi_{\mu}(\omega) = \E_{\beta \sim \mu} \log \frac{1}{\abs{\beta - \omega}} = \E_{t \sim \gamma} \log \frac{1}{\abs{L(t) - L(z)}}
		=
		\phi_{\gamma}(z) + \log \frac{2}{M-m}\,,
	\end{align*}
	where $z := L^{-1}(\omega)$. 
	This gives the desired identity:
	\begin{align*}
		\E_{\beta \sim \mu} \log \abs{1 - \beta^{-1} \lambda }
		= \E_{\beta \sim \mu} \log \frac{1}{|\beta|} - \E_{\beta \sim \mu} \log \frac{1}{\abs{\lambda - \beta}}
		= \phi_{\mu}(0) - \phi_{\mu}(\lambda) 
		=
		- \log \left( \abs{z} + \sqrt{z^2 - 1}\right)
		= \log \frac{\sqrt{\kappa} - 1}{\sqrt{\kappa} + 1}\,.
	\end{align*}
	Above the first step is by expanding the logarithm, the second step is by definition of the logarithmic potential $\phi_{\mu}$, the third step is by Fact~\ref{fact:eq} and the above formula for $\phi_{\mu}$ in terms of $\phi_{\gamma}$, and the last step is by plugging in the value of $z = L^{-1}(0) = -\tfrac{\kap + 1}{\kap - 1}$ and simplifying.
\end{proof}

\begin{remark}[Connection to Green's function]
	This expression $\log \frac{\sqrt{\kappa} - 1}{\sqrt{\kappa} + 1} = \log \Racc$ is the evaluation at $0$ of Green's function for the interval $[m,M]$. 	This connection essentially arises because: 1) for quadratic optimization, polynomial approximation dictates the convergence rate of GD (see \S\ref{ssec:hi:quad}), and 2) the Bernstein-Walsh Theorem characterizes the optimal rate of convergence of polynomial approximations in terms of Green's function (see e.g.,~\citep[Theorem 1]{6steps}).
\end{remark}

\subsubsection{Proof of Lemma~\ref{lem:arc-extremal}}\label{app:arc-extremal}

Here we prove that the Arcsine distribution is the unique optimal solution for the stepsize problem~\eqref{eq:extremal-app} by mirroring the standard argument from potential theory that the Arcsine distribution is the unique optimal solution to the electrostatic problem~\eqref{eq:extremal-pot}, see e.g.,~\citep[Example 2.9]{saff10}. Specifically, our approach is to show that any optimal distribution $\mu$ satisfies the equalization property (see \S\ref{app:pot-equalization-stepsize} for a heuristic argument of this). Indeed, it then follows that the optimal distribution $\mu$ is unique and is the Arcsine distribution by the converse of Frostman's Theorem (stated below), which ensures that the Arcsine distribution is the unique distribution satisfying equalization.

We begin with a simple observation.

\begin{obs}[Support of the optimal stepsize distribution]\label{obs:constrained}
    The constraint $\mu \in \cM(E)$ does not change the value of the stepsize problem~\eqref{eq:extremal-app}.
\end{obs}
\begin{proof}
    It suffices to argue that for any distribution $\mu$ (with arbitrary support), there exists some distribution $\tilde{\mu} \in \cM(E)$ whose objective value is at least as good, i.e., 
    $\max_{\lambda \in [m,M]} \E_{\tilde{\beta} \sim \tilde{\mu}} \log |1 - \tilde{\beta}^{-1} \lambda| 
        \leq
        \max_{\lambda \in [m,M]} \E_{\beta \sim \mu} \log |1 - \beta^{-1} \lambda|$.
    To this end, it suffices to observe the elementary inequality
    \[
        |1 - \lambda \tilde{\alpha}| \leq |1 - \lambda \alpha|\,, \qquad \forall \lambda \in [m,M]\,,
    \]
   where $\alpha := \beta^{-1} = a + bi \in \C$ is arbitrary and $\tilde{\alpha} := \max(\min(a,M^{-1}),m^{-1})$ is the projection of $\alpha$ onto the interval $[M^{-1},m^{-1}]$. Indeed, this pointwise inequality implies that we can project $\mu$ onto $\cM(E)$ in this way without worsening the objective value of $\mu$. 
\end{proof}

We now show that any optimal distribution, constrained to this interval, must satisfy the equalization property. For this it is helpful to recall several classical facts from potential theory. These can be found, e.g., in page 171 and Theorems 3.2, 2.8, 2.2, and 3.1, respectively, of~\citep{saff10}.

\begin{lemma}[Basic properties of logarithmic potentials]\label{lem:quad-infinite:pt:potential-harmonic}
    If $\mu \in \cM(\C)$, then $\phi_{\mu}$ is lower semicontinuous on $\C$ and harmonic outside the support of $\mu$. 
\end{lemma}

\begin{lemma}[Continuity of the equilibrium potential]\label{lem:quad-infinite:pt:continuity-principle}
	Consider the setting of Theorem~\ref{thm:pt:frostman}. Then the potential $\phi_{\mu^*}$ is continuous on $\C$. 
\end{lemma}

\begin{lemma}[Principle of domination]\label{lem:quad-infinite:pt:principle-domination}
    Suppose $\mu, \nu \in \mathcal{M}(\C)$ are compactly supported and that $I(\mu) <  \infty$. If for some constant $c$, $\phi_{\mu}(z) - \phi_{\nu}(z) \leq c$ holds $\mu$-a.e., then it holds for all $z \in \mathbb{C}$. 
\end{lemma}

\begin{lemma}[Maximum principle]\label{lem:quad-infinite:pt:max-principle}
	Suppose $h$ is a harmonic function on a domain $\Omega \subset \C$. If $h$ attains a local maximum in $\Omega$, then $h$ is constant on $\Omega$.
\end{lemma}

\begin{lemma}[Converse to Frostman's Theorem]\label{lem:pt:frostman-converse} Suppose $\mu \in \calM(E)$ satisfies $I(\mu) <\infty$. If $\phi_{\mu}(z) = c$ on $E$ for some constant $c$, then $\mu$ is the equilibrium distribution for $E$.  
\end{lemma}

\begin{proof}[Proof of Lemma~\ref{lem:arc-extremal}]
	For shorthand, let $\mu$ denote the $\ArcmM$ distribution and note that $I(\mu) < \infty$ by Lemma~\ref{lem:arc-equalization}. Suppose $\nu \in \cM([m,M])$ has objective value for the stepsize problem~\eqref{eq:extremal} that is at least as good as $\mu$. Then by the equalization property for $\mu$ (Lemma~\ref{lem:arc-equalization}),
    \[
        \phi_{\nu}(0) - \phi_{\nu}(\lambda) = 
        \E_{\beta \sim \nu} \log \abs{1 - \beta^{-1} \lambda} \leq
        \E_{\beta \sim \mu} \log \abs{1 - \beta^{-1} \lambda} 
        =
        \phi_{\mu}(0) - \phi_{\mu}(\lambda)\,, \qquad \forall \lambda \in [m,M]\,.
    \]
    Re-arranging gives
	\begin{align}
		f(\lambda) \leq c\,,
		  \qquad \forall \lambda \in [m, M],
		\label{eq:prop:arcsine:opt-implies-equalize:1}
	\end{align}
        where $f := \phi_{\mu} - \phi_{\nu}$ and $c :=f(0)$. 
        By the principle of domination (Lemma~\ref{lem:quad-infinite:pt:principle-domination}), this inequality~\eqref{eq:prop:arcsine:opt-implies-equalize:1} holds for all $\lambda \in \C$. Since $0 \in \Omega := \C \setminus [m,M]$, we conclude that $f$ achieves its maximum on $\Omega$ at $0$. But because $f$ is harmonic on $\Omega$ (Lemma~\ref{lem:quad-infinite:pt:potential-harmonic}), the Maximum Principle (Lemma~\ref{lem:quad-infinite:pt:max-principle}) implies that $f$ is constant on $\Omega$, i.e.,
	\begin{align}
		f(\lambda) = c,
		\qquad \forall \lambda \in \mathbb{C} \setminus [m, M]\,.\label{eq:prop:arcsine:opt-implies-equalize:2}
	\end{align}
	Fix any $\lambda \in [m,M]$ and let $\{z_n\}_{n \in \mathbb{N}}$ be any sequence in $\mathbb{C} \setminus [m,M]$ converging to $\lambda$. Note that $\phi_{\nu}(\lambda) \leq \liminf_{n \to \infty} \phi_{\nu}(z_n)$ by lower semicontinuity (Lemma~\ref{lem:quad-infinite:pt:potential-harmonic}), and $\phi_{\mu} (\lambda) = \limn \phi_{\mu} (z_n)$ by continuity of the equilibrium potential (Lemma~\ref{lem:quad-infinite:pt:continuity-principle}). Thus by~\eqref{eq:prop:arcsine:opt-implies-equalize:2},
	\begin{align*}
            f(\lambda) \geq \limsup_{n \to \infty} f(z_n) = c\,, \qquad \forall \lambda \in [m,M].
	\end{align*}
	Combining this with~\eqref{eq:prop:arcsine:opt-implies-equalize:1} shows that
	\begin{align*}
		f(\lambda) = c\,, \qquad \forall \lambda \in [m, M].
	\end{align*}
	Thus $\nu$ has constant logarithmic potential $\phi_{\nu}(\lambda) = \phi_{\mu}(\lambda) - c$ on $[m,M]$, and therefore by the converse to Frostman's Theorem (Lemma~\ref{lem:pt:frostman-converse}), it must be the equilibrium distribution, a.k.a., $\nu = \ArcmM$. We conclude that the Arcsine distribution is the unique optimal solution for the stepsize problem~\eqref{eq:extremal-app}. The optimal value then follows from Lemma~\ref{lem:arc-equalization}. 
\end{proof}

\section{Deferred proofs}\label{app:deferred}

\subsection{Alternative intuitive proof of Lemma~\ref{lem:arc-equalization} without potential theory}\label{app:equalizing-alternative}

For pedagogical reasons and to provide further intuition for the equalizing property (Lemma~\ref{lem:arc-equalization}), here we present a short semi-rigorous proof that uses only elementary calculus (c.f., the rigorous proof using potential theory in Appendix~\ref{app:pot}). Begin by re-parameterizing: $\beta \sim \ArcmM$ has distribution $\beta = c + d \cos \theta$ where $\theta \sim \mathrm{Unif}(0,\pi)$, $c := \tfrac{M+m}{2}$, $d := \tfrac{M-m}{2}$. 
The desired integral is then
\[
    F(\lambda) := \frac{1}{\pi} \int_0^{\pi} \log \left|1 - \frac{\lambda}{c  + d \cos \theta}\right| d\theta\,.
\]

\paragraph*{Step 1: equalization.} To prove $F$ is constant on $[m,M]$, we argue that $F'$ is zero on $(m,M)$. Forgoing technical details about swapping the derivative and integral, $F'$ is calculated to be 
\begin{align}
    F'(\lambda)
    =
   \frac{1}{\pi} \int_0^{\pi} \frac{d}{d\lambda} \log \left|1 - \frac{\lambda}{c  + d \cos \theta}\right| d\theta
   =
   - \frac{1}{\pi} \mathrm{PV} \int_0^{\pi} \frac{d\theta}{c + d\cos \theta - \lambda} \,.
\end{align}
Here the principal value arises from the simple pole at the value of $\theta$ where $ c+ d \cos \theta = \lambda$. This vanishes since it is well-known\footnote{This can be proved with the substitution $t = \tan(\theta/2)$. Indeed, after some algebra, this substitution gives $\mathrm{PV} \int_0^{\pi} \frac{d\theta}{a + b\cos \theta} = \mathrm{PV}\int_0^{\infty} \frac{2dt}{P-Qt^2}$ where $P := a+b$ and $Q := b-a$ and the simple pole is now at $t_0 = \sqrt{\tfrac{P}{Q}}$. The antiderivative $\int \frac{dt}{P-Qt^2} = \frac{1}{2\sqrt{PQ}} \log |\tfrac{\sqrt{P} + \sqrt{Q}t}{\sqrt{P}-\sqrt{Q}t}| := g(t)$, hence the PV integral is $\lim_{\eps \searrow 0} (g(t_0 - \eps) - g(0) + g(\infty) - g(t_0 + \eps)) = 0$.}   that, more generally,
\begin{align*}
    \mathrm{PV} \int_0^{\pi} \frac{d\theta}{a + b\cos \theta}  = 0
\end{align*}
for any $|a| < |b|$ and $b> 0$. (For us, $a = c - \lambda$ and $b = d$.) 

\paragraph*{Step 2: value.} Since $F$ is constant, we can evaluate at the midpoint $\lambda = c = (M+m)/2$. Compute
\begin{align*}
     F(\lambda)
     &= \frac{1}{\pi} \int_0^{\pi} \log \Big| 1 - \frac{c}{c + d \cos \theta} \Big| d \theta
     = 
     \frac{1}{\pi} \int_0^{\pi} \Big[ \log d + \log |\cos \theta| - \log (c + d\cos \theta) \Big] d\theta
     \\ &= 
     \log d - \log 2 - \log \frac{c + \sqrt{c^2-d^2}}{2}
     = \log \frac{\sqrt{\kappa}-1}{\sqrt{\kappa}+1}\,.
\end{align*}
Above, the first step is by definition of $F$, the second step is by expanding the logarithm, and the final step is by collecting the logarithms and simplifying. It remains to justify the three integral calculations in the penultimate step. The first integral is trivial as the integrand $\log d$ is constant in $\theta$; the second and third are due to the standard integral formulas~\citep[4.224(7) and 4.224(9)]{GR}.

\subsection{Equivalence of separability and commutative Hessians}\label{app:sep-commute}

Here we prove an alternative characterization of separability. We remark that this provides a simple way to certify that a function is not separable: it suffices to exhibit points $x$ and $y$ for which the Hessians $\nabla^2 f(x)$ and $\nabla^2 f(y)$ do not commute.

\begin{lemma}[Equivalence of separability and commutative Hessians]\label{lem:sep-commute}
	Let $f \in C^2(\R^d)$.
	Then the following are equivalent:
	\begin{itemize}
		\item [(i)] 
		$f(x) = \sum_{i=1}^d f_i([y]_i)$ where $y = Ux$ for some orthogonal $U$ and functions $f_i : \R \to \R$.
		\item [(ii)] 
		$\{\nabla^2 f(x)\}_{x \in \R^d}$ commute.
	\end{itemize}
\end{lemma}

\par \noindent \emph{Proof of (i) $\Longrightarrow$ (ii).} The function $h(y) := f(U^Ty) = \sum_{i=1}^d f_i([y]_i)$ is separable, hence its Hessian $\nabla^2 h$ is diagonal. The chain rule implies $\nabla^2 f = U^T \nabla^2 h U$, thus the Hessians of $f$ are simultaneously diagonalizable in the basis $U$, hence they commute. 
\\ \\ \noindent \emph{Proof of (ii) $\Longrightarrow$ (i).} Since the Hessians commute, they are simultaneously diagonalizable, i.e., there exists an orthogonal matrix $U$ such that $\nabla^2 f(x) = U^T D(x) U$ for all $x$, where $D(x)$ is diagonal. Defining $h(y) := f(U^Ty)$, it follows by the chain rule that $\nabla^2 h(y) = D(U^Ty)$ is diagonal. Thus all mixed partial derivatives vanish, hence $\tfrac{\partial}{\partial [y]_i}h(y)$ is a function only of $[y]_i$.
Integrating implies $h(y) = \sum_{i=1}^d f_i([y]_i)$ for some functions $f_i$. Thus $f(x) = h(Ux) = \sum_{i=1}^d f_i([Ux]_i)$ as desired.

\subsection{Proof of Lemma~\ref{lem:L1-martingale}}\label{app-pf:lem:L1-martingale}

We first isolate a useful identity: the random variables $Z_t = \log |1 - \beta_t^{-1} \lambda_t|$ defined in \S\ref{ssec:sep:uni} are uncorrelated---despite not being independent. Intuitively, this is due to the equalization property of the Arcsine distribution (from which $\beta_t^{-1}$ are independently drawn) since it ensures that each $Z_t$ has the same (conditional) expectation regardless of the dependencies injected through the non-stationary process $\lambda_t$.

\begin{lemma}\label{lem:cov}
	$\Cov(Z_s, Z_t) = 0$ for all $s \neq t$.
\end{lemma}

\begin{proof}
	Suppose $s < t$ without loss of generality. Then
	$\Cov(Z_s, Z_t)
	=
	\E_{\beta_0, \dots, \beta_t} [(Z_s - \log \Racc) (Z_t - \log \Racc)]
	=
	\E_{\beta_0, \dots, \beta_{t-1}} [ \E_{\beta_t} [ 
	(Z_s - \log \Racc) (Z_t - \log \Racc) \, | \, \beta_0, \dots, \beta_{t-1} ]]
	=
	0$, where 
	the first step is by the definition of covariance and~\eqref{eq:sep:uni:main:expectation}, the second step is by the Law of Iterated Expectations, and the third step is because $\E_{\beta_t}[Z_t \; | \; \beta_0, \dots, \beta_{t-1}] = \log \Racc$ by the equalization property of the Arcsine distribution (Lemma~\ref{lem:arc-equalization}).
\end{proof}

\begin{proof}[Proof of Lemma~\ref{lem:L1-martingale}]
	First we check that $\{M_n\}_{n \in \mathbb{N}}$ is indeed a martingale:
	$$
	\E[M_{n} \; | \; F_{n-1}]
	=
	M_{n-1} + \frac{\E\left[Z_{n-1} - \E[Z_{n-1}] \; \Big| \; F_{n-1} \right]}{n}
	=
	M_{n-1}.
	$$
	The final step uses $\E[Z_{n-1} | F_{n-1} ] = \E[Z_{n-1}]$, which follows by an identical argument as in~\eqref{eq:sep:uni:main:expectation}.
	\par Next, we show $\{M_n\}_{n \in \mathbb{N}}$ is $L_1$-bounded. By the Cauchy-Schwarz inequality,
	$$
	\E|M_n| \leq \sqrt{\E[M_n^2]} = \sqrt{\Var(M_n)}
	$$
	Thus it suffices to show that $\Var(M_n)$ is uniformly bounded. To this end, since $\Cov(Z_s, Z_t) = 0$ for any $s \neq t$ by Lemma~\ref{lem:cov}, it follows that
	$$
	\Var(M_n)
	=
	\Var\left(\sum_{t=0}^{n-1} \frac{Z_t}{t+1}\right)
	=
	\sum_{t=0}^{n-1} \Var\left(\frac{Z_t}{t+1}\right)
	\leq
	V\sum_{t=1}^{n}  \frac{1}{t^2}
	\leq 
	V\sum_{t=1}^{\infty} \frac{1}{t^2}
	=
	V \frac{\pi^2}{6}
	$$
	where $V := \sup_{\lambda \in [m, M]} \Var_{\beta \sim \ArcmM}(\log \abs{ 1 - \beta^{-1}\lambda})$ is a uniform upper bound on the variance of $Z_t$. Note that $V < \infty$ because the logarithmic potential is an integrable singularity (see also Appendix 5 of~\citep{kalousek} for quantitative bounds on $V$). Thus $M_n$ has finite variance, as desired.
\end{proof}

    \subsection{Details for \S\ref{ssec:discussion:notions}}\label{app:notions}

	Here we observe that random stepsizes can make GD unstable in alternative notions of convergence. 
	Instability in the following lemma occurs for $\kappa > 4$ since then $\sqrt{\kappa} - 1 > 1$. 
	See \S\ref{ssec:discussion:notions} for a discussion of this phenomenon and why the notion of convergence studied in the rest of this paper is more relevant to optimization practice.
	
	\begin{lemma}[Instability in alternative notions of convergence]\label{lem:stability}
		Consider GD with i.i.d. inverse stepsizes from $\ArcmM$. There is an $m$-strongly convex, $M$-smooth, quadratic function $f$ such that for any initialization $x_0 \neq x^*$ and any number of iterations $n$,
		\begin{align}
			\E\Bigg[\frac{\|x_{n} - x^*\|}{\|x_0 - x^*\|} \Bigg]^{1/n} \geq
			\sqrt{\kappa} - 1\,.
			\label{eq:notion-bad}
		\end{align}
	\end{lemma}

	To prove this, we make use of the following integral representation of the geometric mean of two positive scalars $a, b > 0$ in terms of the weighted harmonic mean $H_{t}(a,b) := (t a^{-1} + (1-t )b^{-1})^{-1}$. This identity is well-known (e.g., it follows from~\citep[Exercise V.4.20]{bhatia2013matrix}, 
    or directly from \citep[\S3.121-2]{GR} by taking $q = 1/b$ and $p = 1/b-1/a$).

	\begin{lemma}[Arcsine integral representation of the geometric mean]\label{lem:whm}
		For any $a, b > 0$,
		\begin{align*}
			\E_{t \sim \text{\emph{Arcsine(0,1)}}} \big[ H_{t}(a,b) \big] = \sqrt{ab}\,.
		\end{align*}
	\end{lemma}

	This lets us compute the expectation of the proposed random stepsizes.
	
	\begin{lemma}[Expectation of proposed random stepsizes]\label{lem:inverse}
		For any $0 < m < M < \infty$,
		\begin{align*}
			\E_{\beta \sim \ArcmM} \big[ \beta^{-1} \big] 
			= \frac{1}{\sqrt{Mm}} \,.
		\end{align*}
	\end{lemma}
	\begin{proof} 
		Re-parameterize $\beta = t M + (1-t) m$ so that $\beta^{-1} = H_t(M^{-1}, m^{-1})$. Since $\beta \sim \mathrm{Arcsine}(m,M)$ implies $t \sim \mathrm{Arcsine}(0,1)$, applying Lemma~\ref{lem:whm} provides the desired identity
		\begin{align*}
			\E_{\beta \sim \ArcmM} \big[ \beta^{-1} \big]
			= \E_{t \sim \mathrm{Arcsine}(0,1)} \big[ H_{t}(M^{-1}, m^{-1}) \big]
			=
			\frac{1}{\sqrt{Mm}}\,.
		\end{align*}
	\end{proof}

	We are now ready to prove Lemma~\ref{lem:stability}. 
 
    \begin{proof}[Proof of Lemma~\ref{lem:stability}]
        Let $\beta_t$ denote the inverse stepsizes; recall that these are i.i.d.\ $\ArcmM$. By expanding the definition of GD for the function $f(x) = \tfrac{M}{2}x^2$,
	\begin{align*}
		x_{n} = \prod_{t=0}^{n-1} (1 - \beta_t^{-1} M) x_0
		\,.
	\end{align*}
	By taking expectations, using independence of the stepsizes, and applying Lemma~\ref{lem:inverse},
	\begin{align*}
		\E [x_{n}] = (1 - \sqrt{\kappa})^n x_0
		\,.
	\end{align*}
	Thus by Jensen's inequality and then the fact that $\kappa > 1$,
	we conclude the desired inequality
	\[
	\E \Bigg[  \frac{\|x_{n} - x^*\|}{\|x_0 - x^*\|}  \Bigg]^{1/n}
	\geq
	\Bigg( \frac{\| \E x_{n} - x^*\|}{\|x_0 - x^*\|}  \Bigg)^{1/n}\
	= \sqrt{\kappa} - 1.
	\]
    \end{proof}

\subsection{Proof of Theorem~\ref{thm:sep:hp}}\label{app:hp}

For shorthand, let $V := \sup_{\lambda \in [m,M]} \Var_{\beta \sim \ArcmM} \left( \log \abs{1 - \beta^{-1} \lambda} \right)$ denote a uniform upper bound on the variance of any $Z_t$. By the integral calculations in~\citep[Appendix 5]{kalousek},
\begin{align*}
	V \leq \pi^2 + o_{\kappa}(1)\,.
\end{align*}
Thus because $\Cov(Z_s, Z_t) = 0$ for $s \neq t$ by Lemma~\ref{lem:cov}, 
\begin{align*}
	\Var\left( \frac{1}{n} \sum_{t=0}^{n-1} Z_t \right)
	=
	\frac{1}{n^2}\sum_{t=0}^{n-1} \Var\left( Z_t \right) 
	\leq 
	\frac{V}{n}
	\leq \frac{\pi^2 + o_{\kappa}(1)}{n}\,.
\end{align*}
Thus by Chebyshev's inequality and the equalization property of the Arcsine distribution (Lemma~\ref{lem:arc-equalization}),
\begin{align}
	\Prob\left( \frac{1}{n} \sum_{t=0}^{n-1} \log|1 - \beta_t^{-1} \lambda_t| \geq \log \Racc + \eps \right) \leq \frac{\pi^2 + o_{\kappa}(1)}{n \eps^2}\,,
\end{align}
where $\beta_t$ are i.i.d.\ $\ArcmM$ and $\lambda_t \in [m,M]$ is any process adapted to the filtration $\sigma(\{\beta_s\}_{s < t})$.
Thus by~\eqref{eq:rate-main-proof}, the average convergence rate 
$R_{n,i} := (\tfrac{| [x_{n}]_i - [x^*]_i | }{ | [x_0]_i - [x^*]_i|})^{1/n}$
for the function $f_i$ satisfies
\begin{align}
	\Prob\left( R_{n,i} \geq e^{\eps} \, \Racc \right) \leq \frac{\pi^2 + o_{\kappa}(1)}{n \eps^2}
\end{align}
for any coordinate $i \in S := \{i \in [d] : [x_0]_i \neq [x^*]_i\}$, i.e., for any $i$ for which the initialization is not already optimal. Thus the overall convergence rate
$R_n := (\tfrac{\|x_{n} - x^*\| }{ \|x_0 - x^*\|})^{1/n}$
for $f$ satisfies
\begin{align}
	\Prob\left(R_n \geq e^{\eps} \, \Racc \right) \leq
	\Prob\left(\max_{i \in S} R_{n,i} \geq e^{\eps} \, \Racc \right)
	\leq 
	\frac{d(\pi^2 + o_{\kappa}(1))}{n \eps^2}\,.
\end{align}
Above the first step is because $R_n \leq \max_{i \in S} R_{n,i}$ (since $R_n^2 = \tfrac{\sum_{i \in S} |[x_{n}]_i - [x^*]_i|^2}{\sum_{i \in S} |[x_{0}]_i - [x^*]_i|^2} \leq \max_{i \in S} \tfrac{|[x_{n}]_i - [x^*]_i|^2}{ |[x_{0}]_i - [x^*]_i|^2} = \max_{i \in S} R_{n,i}^2$). The second step is by a union bound and the previous display. The proof is then complete by setting $\delta$ equal to the right hand side of the above display.

	\subsection{Details for \S\ref{ssec:discussion:inexact}}\label{app:inexact}
	\begin{proof}[Proof of Theorem~\ref{thm:inexact}]
		The proof is similar to that in \S\ref{ssec:sep:uni}, so we just highlight the differences. Denote $\tilde{\lambda}_t := \tilde{g}_t / (x_t - x^*)$. Then
		\[
		x_{t+1} - x^*
		= x_t - \alpha_t \tilde{g}_t - x^*
		= (1 - \alpha_t \tilde{\lambda}_t) (x_t - x^*)\,,
		\]
		thus 
		\[
		x_n - x^* = \prod_{t=0}^{n-1} (1 - \alpha_t \tilde{\lambda}_t) (x_0 - x^*)\,,
		\]
		and so in particular 
		\[
		\log \left( \frac{|x_n - x^*|}{|x_0 - x^*|}\right)^{1/n}
		= 
		\frac{1}{n} \sum_{t=0}^{n-1} \log \abs{ 1 - \alpha_t \tilde{\lambda}_t }\,. 
		\]
        Now by the multiplicative error assumption~\eqref{eq:inexact} and Lemma~\ref{lem:sep:uni:helper}, $\tilde{m} \leq \tilde{\lambda}_t \leq \tilde{M}$ where $\tilde{m} := (1 - \eps)m$ and $\tilde{M} := (1+\eps)M$. 
		Thus by a similar martingale-based Law of Large Numbers argument as in \S\ref{ssec:sep:uni}, it suffices to compute
		\begin{align}
                      \max_{\tilde{\lambda} \in [\tilde{m},\tilde{M}]} \E_{\beta \sim \ArcmM} \log \abs{1 - \beta^{-1}\tilde{\lambda}}
			\label{eq:pf-inexact:1}
		\end{align}
		and in particular to show that this is equal to the logarithm of the upper bound in~\eqref{eq:thm-inexact}. By the same calculation as in Appendix~\ref{app:arc-equalization}, this expectation simplifies to
		\begin{align}
			\E_{\beta \sim \ArcmM} \log \abs{1 - \beta^{-1} \tilde{\lambda}} = \log \Racc - \log (|z| - \sqrt{z^2 - 1}) \cdot \mathds{1}_{z \notin [-1,1]}\,,
		\end{align}
		where $z := \tfrac{2}{M-m} (\tilde{\lambda} - \tfrac{M+m}{2})$.
        This quantity is increasing in $|z|$, thus~\eqref{eq:pf-inexact:1} is maximized at $\tilde{\lambda} = \tilde{M}$ which corresponds to $|z| = 1+\tfrac{2M\eps}{M-m} = 1+ \tilde{\eps}$. Plugging this into~\eqref{eq:pf-inexact:1} yields the desired bound.
		\par Finally, for tightness, note that the analysis is tight at the aforementioned extreme $\tilde{\lambda} = \tilde{M}$ which corresponds to $f(x) = \tfrac{Mx^2}{2}$, with estimates $\tilde{g}_t$ that maximally overestimate the true gradient.
	\end{proof}

    \begin{proof}[Proof of Theorem~\ref{thm:inexact-reoptimized}]
        By~\eqref{eq:pf-inexact:1} in the proof of  Theorem~\ref{thm:inexact}, the worst-case convergence rate for $m$-strongly convex and $M$-smooth functions with $\eps$-inexact gradients is governed by the convergence rate for $\tilde{m}$-strongly convex and $\tilde{M}$-smooth functions with exact gradients. Thus by Lemmas~\ref{lem:arc-extremal} and~\ref{lem:arc-equalization}, the unique optimal distribution for inverse stepsizes is $\Arc{\tilde{m}}{\tilde{M}}$ which achieves the rate $(\sqrt{\tilde{\kappa}}-1)/(\sqrt{\tilde{\kappa}}+1)$. Tightness follows by a similar argument as in the proof of Theorem~\ref{thm:inexact}. 
    \end{proof}

    \footnotesize
	\addcontentsline{toc}{section}{References}
	\bibliographystyle{plainnat}
	\bibliography{hedging}{}

\end{document}